\documentclass{amsart}
\usepackage{amsmath,amssymb,amsthm,amscd,latexsym}
\usepackage[dvipdfmx]{graphicx}
\usepackage[all]{xy}
\usepackage{wrapfig}
\usepackage{color}
\usepackage{enumerate}
\usepackage{cases}
\usepackage[top=30mm,bottom=30mm,left=30mm,right=30mm]{geometry}
\usepackage{graphics}
\usepackage{tikz}
\usetikzlibrary{patterns}
\usepackage{pxpgfmark}
\usepackage{multirow}
\usepackage{comment}
\usetikzlibrary{intersections, calc,decorations.markings}

\numberwithin{equation}{section}
\theoremstyle{definition}
\newtheorem{example}{Example}[section]
\newtheorem{definition}[example]{Definition}

\newtheorem{remark}[example]{Remark}

\theoremstyle{plain}
\newtheorem{lemma}[example]{Lemma}
\newtheorem{theorem}[example]{Theorem}
\newtheorem{proposition}[example]{Proposition}
\newtheorem{corollary}[example]{Corollary}

\renewcommand{\mod}{{\rm mod}}
\renewcommand{\dim}{{\rm dim}}

\DeclareMathOperator{\udim}{\underline{\rm dim}}
\DeclareMathOperator{\End}{{\rm End}}
\DeclareMathOperator{\Hom}{{\rm Hom}}

\DeclareMathOperator{\rigid}{{\rm rigid}}
\DeclareMathOperator{\irigid}{{\rm irigid}}

\DeclareMathOperator{\clv}{{\rm cl-var}}
\DeclareMathOperator{\clus}{{\rm cluster}}
\DeclareMathOperator{\Int}{{\rm Int}}

\DeclareMathOperator{\bA}{\mathbb{A}}
\DeclareMathOperator{\bQ}{\mathbb{Q}}

\DeclareMathOperator{\bT}{\mathbb{T}}
\DeclareMathOperator{\bZ}{\mathbb{Z}}

\DeclareMathOperator{\cA}{\mathcal{A}}
\DeclareMathOperator{\cC}{\mathcal{C}}
\DeclareMathOperator{\cD}{\mathcal{D}}
\DeclareMathOperator{\cF}{\mathcal{F}}
\DeclareMathOperator{\cS}{\mathcal{S}}

\DeclareMathOperator{\sC}{\mathsf{C}}

\DeclareMathOperator{\sM}{\mathsf{M}}
\DeclareMathOperator{\sX}{\mathsf{X}}
\DeclareMathOperator{\cl}{\mathsf{c}}

\title{Denominator vectors and dimension vectors from triangulated surfaces}
\author{Toshiya Yurikusa}
\address{Laboratoire de Math\'{e}matiques de Versailles, UVSQ, CNRS, Universit\'{e} Paris-Saclay, 78035 Versailles, France and Mathematical Institute, Tohoku University, Sendai 980-8578, Japan}
\email{toshiya.yurikusa.d8@tohoku.ac.jp}

\begin{document}
\maketitle

\begin{abstract}
In a categorification of skew-symmetric cluster algebras, each cluster variable corresponds with an indecomposable module over the associated Jacobian algebra. Buan, Marsh and Reiten studied  when the denominator vector of each cluster variable in an acyclic cluster algebra coincides with the dimension vector of the corresponding module. In this paper, we give analogues of their results for cluster algebras from triangulated surfaces by comparing two kinds of intersection numbers of tagged arcs.
\end{abstract}

\section{Introduction}

Cluster algebras \cite{FZ02} are commutative algebras with generators called cluster variables. The certain tuples of cluster variables are called clusters. In a cluster algebra with principal coefficients, by Laurent phenomenon, any cluster variable $x$ is expressed by a Laurent polynomial of the initial cluster variables $(x_1,\ldots,x_n)$ and coefficients $(y_1,\ldots,y_n)$
\[
x=\frac{F(x_1, \ldots,x_n,y_1,\ldots,y_n)}{x_1^{d_1(x)} \cdots x_n^{d_n(x)}},
\]
where $d_i(x) \in \bZ$ and $F(x_1, \ldots, x_n,y_1,\ldots,y_n) \in \bZ[x_1,\ldots, x_n,y_1,\ldots,y_n]$ is not divisible by any $x_i$ \cite{FZ02,FZ07}. We call $d(x):=(d_i(x))_{1 \le i \le n}$ the denominator vector of $x$. For the maximal degree $f_i(x)$ of $y_i$ in the $F$-polynomial $F(1,\ldots,1,y_1,\ldots,y_n)$, we call $f(x):=(f_i(x))_{1 \le i \le n}$ the $f$-vector of $x$. It was conjectured that $d(x)=f(x)$ for any non-initial cluster variable $x$ in \cite[Conjecture 7.17]{FZ07}.

Let $Q$ be a quiver without oriented cycles of length at most two and $K$ an algebraically closed field. The quiver $Q$ defines a cluster algebra $\cA(Q)$ \cite{FZ02} and a Jacobian algebra $J(Q)=J(Q,W)$ over $K$ for a non-degenerate potential $W$ of $Q$ \cite{DWZ08}. Then non-initial cluster variables $x$ of $\cA(Q)$ correspond with certain indecomposable $J(Q)$-modules $\sM(x)$ (see Subsection \ref{subsec:categorification}). Note that the dimension vector $\udim\sM(x)$ is independent of the choice of $W$ as stated below. It was proved in \cite{CK06} (see also \cite{CC06,H06}) that if $Q$ is acyclic, then it satisfies the property
\begin{equation*}\label{eq:d=dim}
\tag{d=dim} \text{$d(x)=\udim\sM(x)$ for any non-initial cluster variable $x$ of $\cA(Q)$}.
\end{equation*}
However, it does not hold in general though it is known that $d(x)\le\udim\sM(x)$ \cite[Corollary 5.5]{DWZ10}. Buan, Marsh, and Reiten \cite{BMR09} gave a characterization of quivers satisfying \eqref{eq:d=dim} for mutation equivalence classes of acyclic quivers (see also \cite[Theorem 1.1]{GLS16}).

\begin{theorem}[{\cite[Theorems 1.5 and 1.6]{BMR09}}]\label{thm:BMR1}
Let $Q$ be a quiver obtained from an acyclic quiver by a sequence of mutations. If $Q$ satisfies \eqref{eq:d=dim}, then $\End_{J(Q)}(P)\simeq K$ for any indecomposable projective $J(Q)$-module $P$. In addition, if $J(Q)$ is a tame algebra, then the inverse holds.
\end{theorem}

Note that \cite[Theorems 1.5 and 1.6]{BMR09} were given in terms of the associated cluster category, but they are equivalent to Theorem \ref{thm:BMR1} by \cite[Remark 3.4]{BMR09} (see \eqref{eq:End}). They also gave the following result as a consequence of Theorem \ref{thm:BMR1}.

\begin{theorem}[{\cite[Corollary 1.7]{BMR09}}]\label{thm:BMR2}
Let $Q$ be an acyclic quiver. Then every quiver obtained from $Q$ by a sequence of mutations satisfies \eqref{eq:d=dim} if and only if $Q$ is a Dynkin quiver or has exactly two vertices.
\end{theorem}

On the other hand, in the case that $J(Q)$ is finite dimensional, Fu and Keller \cite{FK10} proved that the equality $f(x)=\udim\sM(x)$ holds and it gave a counterexample of \cite[Conjecture 7.17]{FZ07} (based on \cite{BMR09}). This equality holds in general (Theorem \ref{thm:bijection-sC}), thus the vector $\udim\sM(x)$ and the property \eqref{eq:d=dim} are independent of the choice of $W$ since $f$-vectors are so.

In this paper, we consider Theorems \ref{thm:BMR1} and \ref{thm:BMR2} for quivers defined from triangulated surfaces. Cluster algebras associated with marked surfaces were introduced in \cite{FoST08} based on \cite{FG06}. To a tagged triangulation $T$ of a marked surface $\cS$, they associated a quiver $Q_T$ and studied the cluster algebra $\cA(Q_T)$. In particular, tagged arcs (resp., tagged triangulations, tagged arcs of $T$) of $\cS$ correspond bijectively with cluster variables (resp., clusters, the initial cluster variables) of $\cA(Q_T)$ except that $\cS$ is a closed surface with exactly one puncture. Moreover, flips of tagged triangulations are compatible with mutations of clusters (see Section \ref{sec:CA from TS}).

Moreover, there are two kinds of intersection numbers of tagged arcs: One was introduced in \cite{FoST08} to give denominator vectors in $\cA(Q_T)$; Another one was introduced in \cite{QZ17} to give dimension of homomorphisms in the associated cluster category. Then the latter also gives $f$-vectors in $\cA(Q_T)$ (Theorem \ref{thm:bijection-x}). Therefore, we mainly study these intersection numbers. As a consequence, we characterize quivers satisfying \eqref{eq:d=dim} in geometric terms.

\begin{theorem}\label{thm:d=dim}
Let $\cS$ be a marked surface and $T$ a tagged triangulation of $\cS$. Then $Q_T$ satisfies \eqref{eq:d=dim} if and only if either $\cS$ is a torus with exactly one puncture or $T$ has neither loops nor tagged arcs connecting punctures.
\end{theorem}

\begin{theorem}\label{thm:d=dim cS}
Let $\cS$ be a marked surface.
\begin{enumerate}
\item There is at least one tagged triangulation $T$ of $\cS$ such that $Q_T$ satisfies \eqref{eq:d=dim} if and only if $\cS$ is one of the following:
\begin{itemize}
 \item a marked surface of genus zero with non-empty boundary;
 \item a marked surface of positive genus with non-empty boundary and at least three marked points, at least two of which are on $\partial S$;
 \item a torus with exactly one puncture.
\end{itemize}
\item The quivers $Q_T$ satisfy \eqref{eq:d=dim} for all tagged triangulations $T$ of $\cS$ if and only if $\cS$ is one of the following:
\begin{itemize}
 \item a polygon with at most one puncture;
 \item an annulus with exactly two marked points;
 \item a torus with exactly one puncture.
\end{itemize}
\end{enumerate}
\end{theorem}

An analogue of Theorem \ref{thm:BMR2} immediately follows from Theorem \ref{thm:d=dim cS}(2) (see Example \ref{ex:quivers from sS}), where we say that a quiver is Dynkin type if it is obtained from a Dynkin quiver by a sequence of mutations.

\begin{corollary}\label{cor:analogue2}
Let $T$ be a tagged triangulation. Then every quiver obtained from $Q_T$ by a sequence of mutations satisfies \eqref{eq:d=dim} if and only if $Q_T$ is either Dynkin type, a Kronecker quiver, or a Markov quiver.
\end{corollary}

Moreover, we also give an analogue of Theorem \ref{thm:BMR1} except for closed surfaces with exactly one puncture by using Theorem \ref{thm:d=dim}.

\begin{theorem}[Theorem \ref{thm:TFAE final}]\label{thm:analogue1}
Let $T$ be a tagged triangulation of a marked surface which is not a closed surface with exactly one puncture. Then $Q_T$ satisfies \eqref{eq:d=dim} if and only if $\End_{J(Q_T)}(P)\simeq K$ for any indecomposable projective $J(Q_T)$-module $P$.
\end{theorem}

Finally, if $\cS$ is a torus with exactly one puncture, then $Q_T$ satisfies \eqref{eq:d=dim} by Theorem \ref{thm:d=dim}. However, Labardini-Fragoso \cite{Lab09} constructed a non-degenerate potential $W$ of $Q_T$ such that 
\begin{itemize}
\item $\dim\End_{J(Q_T,W)}(P)=4$ for any indecomposable projective $J(Q_T,W)$-module $P$ \cite[Proposition 4.2]{Lad12};
\item $J(Q_T,W)$ is tame \cite{GLS16} (see also \cite[Theorem 3.5]{Lab16survey}).
\end{itemize} 
Therefore, Theorem \ref{thm:BMR1} (or Theorem \ref{thm:analogue1}) does not hold in general.

This paper is organized as follows. In Section \ref{sec:cS}, we only work on marked surfaces. We study two kinds of intersection numbers and give geometric versions of Theorems \ref{thm:d=dim} and \ref{thm:d=dim cS} (Theorems \ref{thm:CIV}, \ref{thm:wCIV} and \ref{thm:CIV cS}). In Section \ref{sec:categorification}, we recall cluster algebras and their categorification. For the case arising from marked surfaces, they have many properties as in Section \ref{sec:CA from TS}. Using the properties and the results in Section \ref{sec:cS}, we give Theorems \ref{thm:d=dim}, \ref{thm:d=dim cS} in Subsection \ref{subsec:CA TS} and Theorem \ref{thm:analogue1} in Subsection \ref{subsec:categorification TS}. Moreover, we also characterize the condition of Corollary \ref{cor:analogue2} in terms of the associated cluster category (Proposition \ref{prop:characterization}). In Section \ref{sec:example}, we give an example which is not covered by \cite{BMR09}.

\section{On marked surfaces}\label{sec:cS}

In this section, we work on marked surfaces \cite{FoST08,FT18}.

\subsection{Intersection numbers of tagged arcs}\label{subsec:int}

Let $S$ be a connected compact oriented Riemann surface with (possibly empty) boundary $\partial S$ and $M$ a non-empty finite set of marked points on $S$ with at least one marked point on each connected component of $\partial S$. We call the pair $\cS=(S,M)$ a \emph{marked surface}. A marked point in the interior of $S$ is called a \emph{puncture}. For technical reasons, we assume that $\cS$ is not a monogon with at most one puncture, a digon without punctures, a triangle without punctures, and a sphere with at most three punctures (see \cite{FoST08} for the details). Throughout this paper, we consider a curve on $S$ up to isotopy relative to $M$.

A \emph{tagged arc} of $\cS$ is a curve on $S$ whose endpoints are in $M$ and each end is tagged in one of two ways, \emph{plain} or \emph{notched}, such that the following conditions are satisfied:
\begin{itemize}
 \item it does not intersect itself except at its endpoints;
 \item it is disjoint from $M$ and $\partial S$ except at its endpoints;
 \item it does not cut out a monogon with at most one puncture or a digon without punctures;
 \item its ends incident to $\partial S$ are tagged plain;
 \item both ends of a loop are tagged in the same way,
\end{itemize}
where a \emph{loop} is a tagged arc with two identical endpoints. In the figures, we represent tags as follows:
\[
\begin{tikzpicture}[baseline=-1mm]
 \coordinate(0)at(0,0) node[left]{plain};
 \coordinate(1)at(1,0); \fill(1)circle(0.07);
 \draw(0)to(1);
\end{tikzpicture}
\hspace{7mm}
\begin{tikzpicture}[baseline=-1mm]
 \coordinate(0)at(0,0) node[left]{notched};
 \coordinate(1)at(1,0); \fill(1)circle(0.07);
 \draw(0)to node[pos=0.8]{\rotatebox{90}{\footnotesize $\bowtie$}}(1);
\end{tikzpicture}\ .
\]
A \emph{pair of conjugate arcs} is the following pair of tagged arcs whose underlying curves coincide:
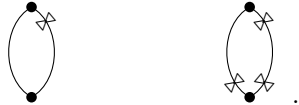
\[
\begin{tikzpicture}
 \coordinate(0)at(0,0);
 \coordinate(1)at(0,-1.2);
 \draw(0)to[out=-60,in=-120,relative](1);
 \draw(0)to[out=60,in=120,relative] node[pos=0.2]{\rotatebox{40}{\footnotesize $\bowtie$}}(1);
 \fill(0)circle(0.07); \fill(1)circle(0.07);
\end{tikzpicture}
\hspace{20mm}
\begin{tikzpicture}
 \coordinate(0)at(0,0);
 \coordinate(1)at(0,-1.2);
 \draw(0)to[out=-60,in=-120,relative] node[pos=0.8]{\rotatebox{40}{\footnotesize $\bowtie$}}(1);
 \draw(0)to[out=60,in=120,relative] node[pos=0.2]{\rotatebox{40}{\footnotesize $\bowtie$}} node[pos=0.8]{\rotatebox{-40}{\footnotesize $\bowtie$}}(1);
 \fill(0)circle(0.07); \fill(1)circle(0.07);
\end{tikzpicture}\ .
\]
We denote by $\bA_{\cS}$ the set of tagged arcs of $\cS$. We also recall a rotation of tagged arcs \cite{BQ15} (see Table \ref{table:Int(g,rho g)}).

\begin{definition}\label{def:rotation}
The \emph{tagged rotation} of $\gamma\in\bA_{\cS}$ is the tagged arc $\rho(\gamma)$ defined as follows:
\begin{itemize}
 \item If $\gamma$ has an endpoint $o$ on a component $C$ of $\partial S$, then $\rho(\gamma)$ is obtained from $\gamma$ by moving $o$ to the next marked point on $C$ in the counterclockwise direction;
 \item If $\gamma$ has an endpoint at a puncture $p$, then $\rho(\gamma)$ is obtained from $\gamma$ by changing its tags at $p$.
\end{itemize}
\end{definition}

Throughout this paper, when we consider intersections of curves, we assume that they intersect transversally in a minimum number of points in $S \setminus M$. Our main subjects are two kinds of intersection numbers of tagged arcs defined in \cite[Definition 8.4]{FoST08} and \cite[Definition 3.3]{QZ17}.

\begin{definition}\label{def:intersection numbers}
Let $\gamma, \delta \in \bA_{\cS}$.
\begin{enumerate}
\item The \emph{FST-intersection number $(\gamma | \delta)$ of $\gamma$ and $\delta$} is defined by $A_{\gamma,\delta}+B_{\gamma,\delta}+C_{\gamma,\delta}+D_{\gamma,\delta}$, where
\begin{itemize}
 \item $A_{\gamma,\delta}$ is the number of intersection points of $\gamma$ and $\delta$ in $S \setminus M$;
 \item $B_{\gamma,\delta}=0$ unless $\gamma$ is a loop at $o \in M$, in which case $B_{\gamma,\delta}$ is the negative of the number of contractible triangles $\triangle$ defined as follows: Suppose that $\delta$ consecutively intersects with $\gamma$ at points $a$ and $b$. Then $\triangle$ consists of the segment of $\delta$ from $a$ to $b$, the segment of $\gamma$ from $o$ to $a$, and the segment of $\gamma$ from $o$ to $b$ (see Figure \ref{fig:B});
 \item $C_{\gamma,\delta}=0$ unless the underlying curves of $\gamma$ and $\delta$ coincide, in which case $C_{\gamma,\delta}=-1$;
 \item $D_{\gamma,\delta}$ is the number of ends of $\delta$ incident to an endpoint $e$ of $\gamma$ whose tags are different from one of $\gamma$ at $e$.
\end{itemize}
\item The \emph{QZ-intersection number $\Int(\gamma,\delta)$ of $\gamma$ and $\delta$} is defined by $A_{\gamma,\delta}+C_{\gamma,\delta}'+D_{\gamma,\delta}'$, where
\begin{itemize}
 \item $C_{\gamma,\delta}'=0$ unless $\gamma$ and $\delta$ form a pair of conjugate arcs, in which case $C_{\gamma,\delta}'=-1$;
 \item $D_{\gamma,\delta}'$ is the number of pairs of an end of $\gamma$ and an end of $\delta$ such that they are incident to a common puncture and their tags are different.
\end{itemize}
\end{enumerate}
\end{definition}

\begin{figure}[ht]
\begin{tikzpicture}[baseline=0mm]
 \coordinate(0)at(0,-1); \node at(0.3,-1){$o$}; \node at(-0.6,-0.5){$a$}; \node at(0.6,-0.5){$b$};
 \node at(0.7,0.5){$\gamma$}; \node at(1.2,-0.3){$\delta$}; \node at(0,-0.6){$\triangle$};
 \draw(0)to[out=-60,in=-90,relative](0,1); \draw(0)to[out=60,in=90,relative](0,1);
 \draw(-1,-0.3)--(1,-0.3);
 \fill(0)circle(0.07); \draw[dotted](0,0.4)circle(0.35);
\end{tikzpicture}
\caption{The triangle $\triangle$ contributes $-1$ to $B$ if it is contractible}
\label{fig:B}
\end{figure}
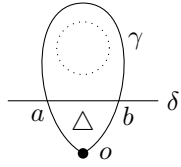

Note that FST-intersection numbers are not symmetric in general, while QZ-intersection numbers are symmetric.

Let $\gamma, \delta \in \bA_{\cS}$. By the definitions, $C_{\gamma,\delta}$, $D_{\gamma,\delta}$, $C_{\gamma,\delta}'$, and $D_{\gamma,\delta}'$ are zero if $\gamma$ and $\delta$ have no common endpoints. Moreover, $C_{\gamma,\delta}$ and $C_{\gamma,\delta}'$ are zero if the underlying curves of $\gamma$ and $\delta$ do not coincide. Thus Table \ref{table:nonzeroBC} lists all cases where any of them are not zero up to changing all tags at each puncture. In particular, $(\gamma | \gamma)=C_{\gamma,\gamma}=-1 \neq 0=C_{\gamma,\gamma}'=\Int(\gamma,\gamma)$. Moreover, if $\gamma$ and $\delta$ form a pair of conjugate arcs, then $(\gamma | \delta)=\Int(\gamma,\delta)=0$.

\renewcommand{\arraystretch}{1.8}
{\begin{table}[ht]
\begin{tabular}{c|c|c|c}
   &
\begin{tikzpicture}[baseline=0mm]
 \coordinate(0)at(0,-0.6); \coordinate(1)at(0,0.6);
 \draw(0)to[out=-60,in=-120,relative] node[fill=white,inner sep=2]{$\delta$}(1);
 \draw(0)to[out=60,in=120,relative] node[fill=white,inner sep=2]{$\gamma$}(1);
 \fill(0)circle(0.07); \fill(1)circle(0.07);
 \node at(0,-0.8) {};
\end{tikzpicture}
   &
\begin{tikzpicture}[baseline=0mm]
 \coordinate(0)at(0,-0.6); \coordinate(1)at(0,0.6);
 \draw(0)to[out=-60,in=-120,relative] node[pos=0.85]{\rotatebox{35}{$\bowtie$}} node[fill=white,inner sep=2]{$\delta$}(1);
 \draw(0)to[out=60,in=120,relative] node[fill=white,inner sep=2]{$\gamma$}(1);
 \fill(0)circle(0.07); \fill(1)circle(0.07);
\end{tikzpicture}
   &
\begin{tikzpicture}[baseline=0mm]
 \coordinate(0)at(0,-0.6); \coordinate(1)at(0,0.6);
 \draw(0)to[out=-60,in=-120,relative] node[pos=0.15]{\rotatebox{-35}{$\bowtie$}} node[pos=0.85]{\rotatebox{35}{$\bowtie$}} node[fill=white,inner sep=2]{$\delta$}(1);
 \draw(0)to[out=60,in=120,relative] node[fill=white,inner sep=2]{$\gamma$}(1);
 \fill(0)circle(0.07); \fill(1)circle(0.07);
\end{tikzpicture}
\\\hline
 $C_{\gamma,\delta}$ & -1 & -1 & -1
\\\hline
 $C'_{\gamma,\delta}$ & 0 & -1 & 0
\\\hline
 $C''_{\gamma,\delta}$ & -1 & 0 & -1
\end{tabular}\vspace{3mm}\\
\begin{tabular}{c|c|c|c|c|c|c}
   &
\begin{tikzpicture}[baseline=0mm]
 \coordinate(0)at(0,-0.6); \coordinate(1)at(0,0.6);
 \draw(0)to[out=-60,in=-120,relative] node[pos=0.85]{\rotatebox{35}{$\bowtie$}} node[fill=white,inner sep=2]{$\delta$}(1);
 \draw(0)to[out=60,in=120,relative] node[fill=white,inner sep=2]{$\gamma$}(1);
 \fill(0)circle(0.07); \fill(1)circle(0.07);
\end{tikzpicture}
   &
\begin{tikzpicture}[baseline=0mm]
 \coordinate(0)at(0,-0.6); \coordinate(1)at(0,0.6);
 \draw(0)to[out=-60,in=-120,relative] node[pos=0.15]{\rotatebox{-35}{$\bowtie$}} node[pos=0.85]{\rotatebox{35}{$\bowtie$}} node[fill=white,inner sep=2]{$\delta$}(1);
 \draw(0)to[out=60,in=120,relative] node[fill=white,inner sep=2]{$\gamma$}(1);
 \fill(0)circle(0.07); \fill(1)circle(0.07);
\end{tikzpicture}
   &
\begin{tikzpicture}[baseline=0mm]
 \coordinate(0)at(0,0);
 \draw(0)--node[pos=0.2]{$\bowtie$} node[left,pos=0.6]{$\delta$}(0,1);
 \draw(0)--node[left,pos=0.6]{$\gamma$}(0,-1);
 \fill(0)circle(0.07);
\end{tikzpicture}
   &
\begin{tikzpicture}[baseline=0mm]
 \coordinate(0)at(0,0);
 \draw(0)--node[pos=0.2]{$\bowtie$} node[left,pos=0.6]{$\delta$}(0,1);
 \draw(0)to[out=-80,in=-90,relative]node[left]{$\gamma$}(0,-1); \draw(0)to[out=80,in=90,relative](0,-1);
 \fill(0)circle(0.07); \draw[dotted](0,-0.5)circle(0.2);
 \node at(0,-1.1) {};
\end{tikzpicture}
   &
\begin{tikzpicture}[baseline=0mm]
 \coordinate(0)at(0,0);
 \draw(0)--node[left,pos=0.6]{$\gamma$} (0,-1);
 \draw(0)to[out=-80,in=-90,relative]node[pos=0.2]{\rotatebox{-40}{$\bowtie$}}(0,1); \draw(0)to[out=80,in=90,relative]node[pos=0.2]{\rotatebox{40}{$\bowtie$}}node[left]{$\delta$}(0,1);
 \fill(0)circle(0.07); \draw[dotted](0,0.55)circle(0.2);
\end{tikzpicture}
   &
\begin{tikzpicture}[baseline=0mm]
 \coordinate(0)at(0,0);
 \draw(0)to[out=-80,in=-90,relative] node[left]{$\gamma$} (0,-1); \draw(0)to[out=80,in=90,relative] (0,-1);
 \draw(0)to[out=-80,in=-90,relative] node[pos=0.2]{\rotatebox{-40}{$\bowtie$}} (0,1); \draw(0)to[out=80,in=90,relative] node[pos=0.2]{\rotatebox{40}{$\bowtie$}} node[left]{$\delta$}(0,1);
 \fill(0)circle(0.07); \draw[dotted](0,-0.5)circle(0.2); \draw[dotted](0,0.55)circle(0.2);
\end{tikzpicture}
\\\hline
 $D_{\gamma,\delta}$ & 1 & 2 & 1 & 1 & 2 & 2
\\\hline
 $D'_{\gamma,\delta}$ & 1 & 2 & 1 & 2 & 2 & 4
\\\hline
 $D''_{\gamma,\delta}$ & 0 & 0 & 0 & -1 & 0 & -2
\end{tabular}\vspace{3mm}
\caption{Cases that any of $C_{\gamma,\delta}$, $D_{\gamma,\delta}$, $C'_{\gamma,\delta}$, and $D'_{\gamma,\delta}$ are not zero}
\label{table:nonzeroBC}
\end{table}}

By the definitions, we can give a difference between the intersection numbers.

\begin{proposition}\label{prop:difference between intersection numbers}
Let $\gamma, \delta \in \bA_{\cS}$. Then we have
\[
(\gamma | \delta) = \Int(\gamma,\delta)+B_{\gamma,\delta}+C''_{\gamma,\delta}+D''_{\gamma,\delta},
\]
where
\begin{itemize}
 \item $B_{\gamma,\delta}$ is defined in Definition \ref{def:intersection numbers}(1);
 \item $C''_{\gamma,\delta}=0$ unless the underlying curves of $\gamma$ and $\delta$ coincide and they do not form a pair of conjugate arcs, in which case $C''_{\gamma,\delta}=-1$;
 \item $D''_{\gamma,\delta}=0$ unless $\gamma$ is a loop at a puncture $p$, in which case $D''_{\gamma,\delta}$ is the negative of the number of ends of $\delta$ incident to $p$ whose tags are different from ones of $\gamma$.
\end{itemize}
\end{proposition}

\begin{proof}
Table \ref{table:nonzeroBC} gives the equalities $C''_{\gamma,\delta}=C_{\gamma,\delta}-C'_{\gamma,\delta}$ and $D''_{\gamma,\delta}=D_{\gamma,\delta}-D'_{\gamma,\delta}$. Otherwise, all of them are zero as the above observation. Therefore, we get the desired equality from Definition \ref{def:intersection numbers}.
\end{proof}

We observe some properties of the values $B_{\gamma,\delta}$, $C''_{\gamma,\delta}$, and $D''_{\gamma,\delta}$.

\begin{lemma}\label{lem:observations}
Let $\gamma, \delta \in \bA_{\cS}$.
\begin{enumerate}
\item The values $B_{\gamma,\delta}$, $C''_{\gamma,\delta}$, and $D''_{\gamma,\delta}$ are non-positive and $(\gamma|\delta)\le\Int(\gamma,\delta)$.
\item If $\gamma$ is not a loop, then $B_{\gamma,\delta}=D''_{\gamma,\delta}=0$.
\item There is a $\gamma'\in\bA_{\cS}\setminus\{\gamma\}$ such that $C''_{\gamma,\gamma'}\neq 0$ if and only if $\gamma$ connects punctures.
\end{enumerate}
\end{lemma}

\begin{proof}
The assertions immediate follow from the definitions.
\end{proof}

\begin{lemma}\label{lem:loop B}
If $\gamma$ is a loop of $\cS$ at $m\in M\cap\partial S$, then there is a $\delta\in\bA_{\cS}$ such that $B_{\gamma,\delta}\neq 0$.
\end{lemma}

\begin{proof}
Let $C$ be a connected component of $\partial S$ containing $m$. If $\gamma$ is homotopic to $C$, then there are other marked points on $C$. We take a loop $\delta$ at such a marked point which is homotopic to $C$. Then we have $B_{\gamma,\delta}\neq 0$ (see the left diagram of Figure \ref{fig:construction Bneq0}).

Suppose that $\gamma$ is not homotopic to $C$. We take $\delta$ as a tagged arc obtained from $\gamma$ by applying the Dehn twist along a simple closed curve which is homotopic to $C$ (see e.g. \cite{Ma16} for the definition of Dehn twists). Then we get $B_{\gamma,\delta}\neq 0$ (see the right diagrams of Figure \ref{fig:construction Bneq0}).
\end{proof}

\begin{figure}[ht]
\begin{tikzpicture}[baseline=0mm]
 \coordinate(l)at(-0.3,0); \coordinate(r)at(0.3,0); \node[left]at(l){$m$};
 \draw[pattern=north east lines](0,0)circle(0.3);
 \draw(l)..controls(-0.5,0.7)and(0.7,0.8)..(0.7,0);
 \draw(l)..controls(-0.5,-0.7)and(0.7,-0.8)..(0.7,0)node[fill=white,inner sep=2]{$\gamma$};
 \draw(r)..controls(0.7,0.8)and(-1,1)..(-1,0);
 \draw(r)..controls(0.7,-0.8)and(-1,-1)..(-1,0)node[fill=white,inner sep=2]{$\delta$};
 \fill(l)circle(0.07); \fill(r)circle(0.07);
\end{tikzpicture}
\hspace{10mm}
\begin{tikzpicture}[baseline=0mm]
 \coordinate(l)at(-0.3,0); \node[left]at(l){$m$};
 \draw[pattern=north east lines](0,0)circle(0.3); \draw[dotted](-1.5,0)circle(0.3);
 \draw(l)..controls(-1.5,1)and(-2,0.5)..(-2,0);
 \draw(l)..controls(-1.5,-1)and(-2,-0.5)..(-2,0)node[fill=white,inner sep=2]{$\gamma$};
 \fill(l)circle(0.07);
 \draw(-0.2,0)circle(0.7); \node[fill=white,inner sep=2]at(-0.2,0.7){$\ell$};
\end{tikzpicture}
$\xrightarrow{\text{Dehn twist along $\ell$}}$
\begin{tikzpicture}[baseline=0mm]
 \coordinate(l)at(-0.3,0); \node[left]at(l){$m$};
 \draw[pattern=north east lines](0,0)circle(0.3); \draw[dotted](-1.5,0)circle(0.3);
 \draw(l)..controls(-1.5,1)and(-2,0.5)..(-2,0);
 \draw(l)..controls(-1.5,-1)and(-2,-0.5)..(-2,0)node[fill=white,inner sep=2]{$\gamma$};
 \draw(l)..controls(-0.4,0.6)and(0.5,0.6)..(0.5,0);
 \draw(0.5,0)..controls(0.5,-0.8)and(-0.9,-0.8)..(-0.9,0);
 \draw(-0.9,0)..controls(-0.9,1)and(-2.3,1)..(-2.3,0);
 \draw(l)..controls(-0.5,0.7)and(0.6,0.7)..(0.6,0);
 \draw(0.6,0)..controls(0.6,-1.1)and(-0.7,-0.7)..(-0.9,-0.6);
 \draw(-0.9,-0.6)..controls(-1.1,-0.8)and(-2.3,-1.1)..(-2.3,0);
 \fill(l)circle(0.07); \node at(0.8,0){$\delta$};
\end{tikzpicture}
\caption{Constructions of $\delta\in\bA_{\cS}$ satisfying $B_{\gamma,\delta}\neq 0$ for a loop $\gamma$ at $m \in M\cap\partial S$}
\label{fig:construction Bneq0}
\end{figure}
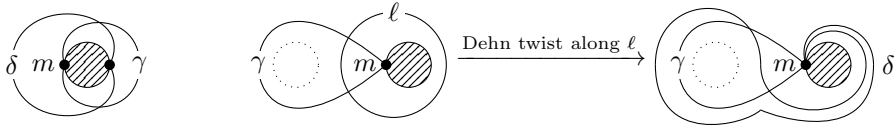

\begin{lemma}\label{lem:once-punctured cases}
Suppose that $\cS$ is a closed surface with exactly one puncture. Let $\gamma\in\bA_{\cS}$.
\begin{enumerate}
\item If the genus of $S$ is one, then we have $B_{\gamma,\delta}=0$ for any $\delta \in \bA_{\cS}$.
\item If the genus of $S$ is two or more, then there is a $\delta\in\bA_{\cS}$ such that $B_{\gamma,\delta}\neq 0$.
\end{enumerate}
\end{lemma}

\begin{proof}
The claim (1) was proved in \cite[Subsection 4.3]{RS17} in the case that all ends of $\gamma$ and $\delta$ are tagged plain. Since $B_{\gamma,\delta}$ is independent of their tags, it holds in general cases.

For (2), we can take a desired $\delta$ concretely. In fact, we consider the case that the genus of $S$ is two. Then $\cS$ is described by the octahedron $P_8$ as follows:
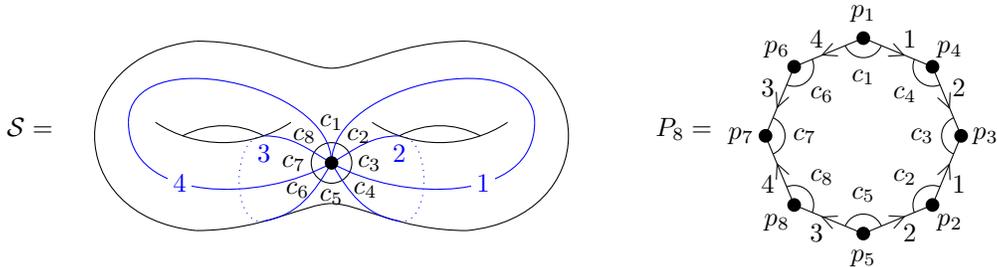
\[
 \cS=
\begin{tikzpicture}[baseline=0mm,scale=1.8]
 \coordinate(u)at(0,0.5);\coordinate(d)at(0,-0.5);
 \coordinate(cl)at(-0.3,0.1);\coordinate(cr)at(0.3,0.1);
 \coordinate(l)at(-0.5,0);\coordinate(r)at(0.5,0);
 \coordinate(lu)at(-1,0.7);\coordinate(ld)at(-1,-0.7);
 \coordinate(ll)at(-1.3,0.1);\coordinate(lr)at(-1.1,0);
 \coordinate(ru)at(1,0.7);\coordinate(rd)at(1,-0.7);
 \coordinate(rr)at(1.3,0.1);\coordinate(rl)at(1.1,0);
 \coordinate(p)at(0,-0.2);
 \draw(u) ..controls(-0.2,0.5)and(-0.5,0.7)..(lu); \draw(lu) ..controls(-2,0.6)and(-2,-0.6)..(ld); \draw(d) ..controls(-0.2,-0.5)and(-0.5,-0.7)..(ld);
 \draw(u) ..controls(0.2,0.5)and(0.5,0.7)..(ru); \draw(ru) ..controls(2,0.6)and(2,-0.6)..(rd); \draw(d) ..controls(0.2,-0.5)and(0.5,-0.7)..(rd);
 \draw(ll) ..controls(-1,-0.1)and(-0.6,-0.1)..(cl);
 \draw(rr) ..controls(1,-0.1)and(0.6,-0.1)..(cr);
 \draw(lr) ..controls(-0.9,0.1)and(-0.7,0.1)..(l);
 \draw(rl) ..controls(0.9,0.1)and(0.7,0.1)..(r);
 \draw(p) circle (1.5mm);
 \draw[blue] (p) to [out=-150,in=-90]node[fill=white,inner sep=2,pos=0.6]{$4$} (-1.5,0.1); \draw[blue] (p) to [out=90,in=90] (-1.5,0.1);
 \draw[blue] (p) to [out=-30,in=-90]node[fill=white,inner sep=2,pos=0.6]{$1$} (1.5,0.1); \draw[blue] (p) to [out=90,in=90] (1.5,0.1);
 \draw[blue] (p) to [out=-120,in=10] (-0.5,-0.63); \draw[blue] (p) to [out=150,in=0] (-0.5,0)node[below]{$3$};
 \draw[blue,dotted] (-0.5,-0.63) to [out=-170,in=180] (-0.5,0);
 \draw[blue] (p) to [out=-60,in=170] (0.5,-0.63); \draw[blue] (p) to [out=30,in=180] (0.5,0)node[below]{$2$};
 \draw[blue,dotted] (0.5,-0.63) to [out=-10,in=0] (0.5,0);
 \fill (p) circle (0.5mm);
 \node at(0,0.1) {$c_1$}; \node at(0.2,0) {$c_2$}; \node at(0.28,-0.2) {$c_3$}; \node at(0.25,-0.4) {$c_4$};
 \node at(0,-0.45) {$c_5$}; \node at(-0.2,0) {$c_8$}; \node at(-0.28,-0.2) {$c_7$}; \node at(-0.25,-0.4) {$c_6$};
\end{tikzpicture}
 \hspace{7mm}
 P_8=
\begin{tikzpicture}[baseline=0mm,scale=1.3]
 \coordinate(0)at(0,0);
 \coordinate(u)at(90:1); \coordinate(d)at(-90:1); \coordinate(r)at(0:1); \coordinate(l)at(180:1);
 \coordinate(ru)at(45:1); \coordinate(rd)at(-45:1); \coordinate(lu)at(135:1); \coordinate(ld)at(-135:1);
 \coordinate(2)at(180:1.5);
 \draw(u)--node{\small\rotatebox{-22.5}{$>$}}node[above right=-1.5]{$1$}(ru)--node{\small\rotatebox{-67.5}{$>$}}node[above right=-1.5]{$2$}(r)--node{\small\rotatebox{67.5}{$>$}}node[below right=-1.5]{$1$}(rd)--node{\small\rotatebox{22.5}{$>$}}node[below right=-1.5]{$2$}(d)--node{\small\rotatebox{-22.5}{$<$}}node[below left=-1.5]{$3$}(ld)--node{\small\rotatebox{-67.5}{$<$}}node[below left=-1.5]{$4$}(l)--node{\small\rotatebox{67.5}{$<$}}node[above left=-1.5]{$3$}(lu)--node{\small\rotatebox{22.5}{$<$}}node[above left=-1.5]{$4$}(u);
 \draw(u)+(-22.5:0.2)arc(-22.5:-157.5:0.2); \draw(ru)+(-67.5:0.2)arc(-67.5:-202.5:0.2); \draw(r)+(-112.5:0.2)arc(-112.5:-247.5:0.2); \draw(rd)+(-157.5:0.2)arc(-157.5:-292.5:0.2);
 \draw(d)+(-202.5:0.2)arc(-202.5:-337.5:0.2); \draw(lu)+(22.5:0.2)arc(22.5:-112.5:0.2); \draw(l)+(67.5:0.2)arc(67.5:-67.5:0.2); \draw(ld)+(112.5:0.2)arc(112.5:-22.5:0.2);
 \node at(90:1.25) {$p_1$}; \node at(45:1.25) {$p_4$}; \node at(0:1.25) {$p_3$}; \node at(-45:1.25) {$p_2$}; \node at(-90:1.25) {$p_5$}; \node at(-135:1.25) {$p_8$}; \node at(-180:1.25) {$p_7$}; \node at(135:1.25) {$p_6$};
 \node at(90:0.6) {$c_1$}; \node at(45:0.6) {$c_4$}; \node at(0:0.6) {$c_3$}; \node at(-45:0.6) {$c_2$}; \node at(-90:0.6) {$c_5$}; \node at(-135:0.6) {$c_8$}; \node at(-180:0.6) {$c_7$}; \node at(135:0.6) {$c_6$};
 \fill(u) circle (0.07); \fill (ru) circle (0.07); \fill(r)circle(0.07); \fill(rd) circle (0.07); \fill(d)circle(0.07); \fill (ld) circle (0.07); \fill(l) circle (0.07); \fill (lu) circle (0.07);
\end{tikzpicture}
\]
where $c_1,\ldots,c_8$ are segments of a simple closed curve which is contracted to a puncture, and the edges of $P_8$ with the same numbers are identified along the directions. In $P_8$, there are the connected segments $\gamma'$ and $\gamma''$ of $\gamma$ containing endpoints of $\gamma$, where they may be identified. We assume that the endpoints of $\gamma'$ and $\gamma''$ are $p_i$ and $p_j$ for $i \le j$, respectively. Then $\gamma'$ and $\gamma''$ intersect with $c_i$ and $c_j$, and let $q_i$ and $q_j$ be their intersection points, respectively. If $\delta$ contains the curve $c_i c_{i+1} \cdots c_j$ or $c_j  \cdots c_n c_1 \cdots c_i$ in $\cS$, then $\delta$ consecutively intersects with $\gamma$ at $q_i$ and $q_j$. Then the segment of $\delta$ from $q_i$ to $q_j$, the segment of $\gamma$ from $p$ to $q_i$, and the segment of $\gamma$ from $p$ to $q_j$ form a contractible triangle, thus we have $B_{\gamma,\delta} \neq 0$. That is, we only need to give such $\delta$. We consider four tagged arcs
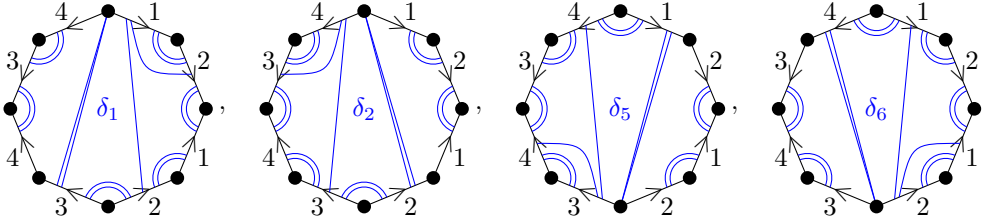
\[
\begin{tikzpicture}[baseline=0mm,scale=1.3]
 \coordinate(0)at(0,0); \node[blue] at(0) {$\delta_1$};
 \coordinate(u)at(90:1); \coordinate(d)at(-90:1); \coordinate(r)at(0:1); \coordinate(l)at(180:1);
 \coordinate(ru)at(45:1); \coordinate(rd)at(-45:1); \coordinate(lu)at(135:1); \coordinate(ld)at(-135:1);
 \coordinate(2)at(180:1.5);
 \draw(u)--node{\small\rotatebox{-22.5}{$>$}}node[above right=-1.5]{$1$}(ru)--node{\small\rotatebox{-67.5}{$>$}}node[above right=-1.5]{$2$}(r)--node{\small\rotatebox{67.5}{$>$}}node[below right=-1.5]{$1$}(rd)--node{\small\rotatebox{22.5}{$>$}}node[below right=-1.5]{$2$}(d)--node{\small\rotatebox{-22.5}{$<$}}node[below left=-1.5]{$3$}(ld)--node{\small\rotatebox{-67.5}{$<$}}node[below left=-1.5]{$4$}(l)--node{\small\rotatebox{67.5}{$<$}}node[above left=-1.5]{$3$}(lu)--node{\small\rotatebox{22.5}{$<$}}node[above left=-1.5]{$4$}(u);
 \draw[blue] (ld)+(-22.5:0.2)--(u); \draw[blue] (ld)+(-22.5:0.25)--(u);
 \draw[blue] (u)+(-22.5:0.2)--($(d)!0.5!(rd)$); \draw[blue] (u)+(-22.5:0.25)..controls(45:0.5)..($(ru)!0.5!(r)$);
 \draw[blue] (ru)+(-67.5:0.2)arc(-67.5:-202.5:0.2); \draw[blue] (r)+(-112.5:0.2)arc(-112.5:-247.5:0.2); \draw[blue] (rd)+(-157.5:0.2)arc(-157.5:-292.5:0.2);
 \draw[blue] (ru)+(-67.5:0.25)arc(-67.5:-202.5:0.25); \draw[blue] (r)+(-112.5:0.25)arc(-112.5:-247.5:0.25); \draw[blue] (rd)+(-157.5:0.25)arc(-157.5:-292.5:0.25);
 \draw[blue] (d)+(-202.5:0.2)arc(-202.5:-337.5:0.2); \draw[blue] (lu)+(22.5:0.2)arc(22.5:-112.5:0.2); \draw[blue] (l)+(67.5:0.2)arc(67.5:-67.5:0.2);
 \draw[blue] (d)+(-202.5:0.25)arc(-202.5:-337.5:0.25); \draw[blue] (lu)+(22.5:0.25)arc(22.5:-112.5:0.25); \draw[blue] (l)+(67.5:0.25)arc(67.5:-67.5:0.25);
 \fill(u) circle (0.07); \fill(ru) circle (0.07); \fill(r) circle (0.07); \fill(rd) circle (0.07); \fill(d) circle (0.07); \fill(ld) circle (0.07); \fill(l) circle (0.07); \fill(lu) circle (0.07);
\end{tikzpicture},
   \hspace{3mm}%
\begin{tikzpicture}[baseline=0mm,scale=1.3]
 \coordinate(0)at(0,0); \node[blue] at(0) {$\delta_2$};
 \coordinate(u)at(90:1); \coordinate(d)at(-90:1); \coordinate(r)at(0:1); \coordinate(l)at(180:1);
 \coordinate(ru)at(45:1); \coordinate(rd)at(-45:1); \coordinate(lu)at(135:1); \coordinate(ld)at(-135:1);
 \coordinate(2)at(180:1.5);
 \draw(u)--node{\small\rotatebox{-22.5}{$>$}}node[above right=-1.5]{$1$}(ru)--node{\small\rotatebox{-67.5}{$>$}}node[above right=-1.5]{$2$}(r)--node{\small\rotatebox{67.5}{$>$}}node[below right=-1.5]{$1$}(rd)--node{\small\rotatebox{22.5}{$>$}}node[below right=-1.5]{$2$}(d)--node{\small\rotatebox{-22.5}{$<$}}node[below left=-1.5]{$3$}(ld)--node{\small\rotatebox{-67.5}{$<$}}node[below left=-1.5]{$4$}(l)--node{\small\rotatebox{67.5}{$<$}}node[above left=-1.5]{$3$}(lu)--node{\small\rotatebox{22.5}{$<$}}node[above left=-1.5]{$4$}(u);
 \draw[blue] (rd)+(-157.5:0.2)--(u); \draw[blue] (rd)+(-157.5:0.25)--(u);
 \draw[blue] (u)+(-157.5:0.2)--($(d)!0.5!(ld)$); \draw[blue] (u)+(-157.5:0.25)..controls(135:0.5)..($(lu)!0.5!(l)$);
 \draw[blue] (ru)+(-67.5:0.2)arc(-67.5:-202.5:0.2); \draw[blue] (r)+(-112.5:0.2)arc(-112.5:-247.5:0.2);
 \draw[blue] (ru)+(-67.5:0.25)arc(-67.5:-202.5:0.25); \draw[blue] (r)+(-112.5:0.25)arc(-112.5:-247.5:0.25);
 \draw[blue] (d)+(-202.5:0.2)arc(-202.5:-337.5:0.2); \draw[blue] (lu)+(22.5:0.2)arc(22.5:-112.5:0.2); \draw[blue] (l)+(67.5:0.2)arc(67.5:-67.5:0.2); \draw[blue] (ld)+(112.5:0.2)arc(112.5:-22.5:0.2);
 \draw[blue] (d)+(-202.5:0.25)arc(-202.5:-337.5:0.25); \draw[blue] (lu)+(22.5:0.25)arc(22.5:-112.5:0.25); \draw[blue] (l)+(67.5:0.25)arc(67.5:-67.5:0.25); \draw[blue] (ld)+(112.5:0.25)arc(112.5:-22.5:0.25);
 \fill(u) circle (0.07); \fill (ru) circle (0.07); \fill(r)circle(0.07); \fill(rd) circle (0.07); \fill(d)circle(0.07); \fill(ld) circle (0.07); \fill(l) circle (0.07); \fill (lu) circle (0.07);
\end{tikzpicture},
   \hspace{3mm}%
\begin{tikzpicture}[baseline=0mm,scale=1.3]
 \coordinate(0)at(0,0); \node[blue] at(0) {$\delta_5$};
 \coordinate(u)at(90:1); \coordinate(d)at(-90:1); \coordinate(r)at(0:1); \coordinate(l)at(180:1);
 \coordinate(ru)at(45:1); \coordinate(rd)at(-45:1); \coordinate(lu)at(135:1); \coordinate(ld)at(-135:1);
 \coordinate(2)at(180:1.5);
 \draw(u)--node{\small\rotatebox{-22.5}{$>$}}node[above right=-1.5]{$1$}(ru)--node{\small\rotatebox{-67.5}{$>$}}node[above right=-1.5]{$2$}(r)--node{\small\rotatebox{67.5}{$>$}}node[below right=-1.5]{$1$}(rd)--node{\small\rotatebox{22.5}{$>$}}node[below right=-1.5]{$2$}(d)--node{\small\rotatebox{-22.5}{$<$}}node[below left=-1.5]{$3$}(ld)--node{\small\rotatebox{-67.5}{$<$}}node[below left=-1.5]{$4$}(l)--node{\small\rotatebox{67.5}{$<$}}node[above left=-1.5]{$3$}(lu)--node{\small\rotatebox{22.5}{$<$}}node[above left=-1.5]{$4$}(u);
 \draw[blue] (ru)+(157.5:0.2)--(d); \draw[blue] (ru)+(157.5:0.25)--(d);
 \draw[blue] (d)+(157.5:0.2)--($(u)!0.5!(lu)$); \draw[blue] (d)+(157.5:0.25)..controls(-135:0.5)..($(ld)!0.5!(l)$);
 \draw[blue] (r)+(-112.5:0.2)arc(-112.5:-247.5:0.2); \draw[blue] (rd)+(-157.5:0.2)arc(-157.5:-292.5:0.2);
 \draw[blue] (r)+(-112.5:0.25)arc(-112.5:-247.5:0.25); \draw[blue] (rd)+(-157.5:0.25)arc(-157.5:-292.5:0.25);
 \draw[blue] (u)+(-22.5:0.2)arc(-22.5:-157.5:0.2); \draw[blue] (lu)+(22.5:0.2)arc(22.5:-112.5:0.2); \draw[blue] (l)+(67.5:0.2)arc(67.5:-67.5:0.2); \draw[blue] (ld)+(-22.5:0.2)arc(-22.5:112.5:0.2);
 \draw[blue] (u)+(-22.5:0.25)arc(-22.5:-157.5:0.25); \draw[blue] (lu)+(22.5:0.25)arc(22.5:-112.5:0.25); \draw[blue] (l)+(67.5:0.25)arc(67.5:-67.5:0.25); \draw[blue] (ld)+(-22.5:0.25)arc(-22.5:112.5:0.25);
 \fill(u) circle (0.07); \fill(ru) circle (0.07); \fill(r) circle (0.07); \fill(rd) circle (0.07); \fill(d) circle (0.07); \fill(ld) circle (0.07); \fill(l) circle (0.07); \fill(lu) circle (0.07);
\end{tikzpicture},
   \hspace{3mm}%
\begin{tikzpicture}[baseline=0mm,scale=1.3]
 \coordinate(0)at(0,0); \node[blue] at(0) {$\delta_6$};
 \coordinate(u)at(90:1); \coordinate(d)at(-90:1); \coordinate(r)at(0:1); \coordinate(l)at(180:1);
 \coordinate(ru)at(45:1); \coordinate(rd)at(-45:1); \coordinate(lu)at(135:1); \coordinate(ld)at(-135:1);
 \coordinate(2)at(180:1.5);
 \draw(u)--node{\small\rotatebox{-22.5}{$>$}}node[above right=-1.5]{$1$}(ru)--node{\small\rotatebox{-67.5}{$>$}}node[above right=-1.5]{$2$}(r)--node{\small\rotatebox{67.5}{$>$}}node[below right=-1.5]{$1$}(rd)--node{\small\rotatebox{22.5}{$>$}}node[below right=-1.5]{$2$}(d)--node{\small\rotatebox{-22.5}{$<$}}node[below left=-1.5]{$3$}(ld)--node{\small\rotatebox{-67.5}{$<$}}node[below left=-1.5]{$4$}(l)--node{\small\rotatebox{67.5}{$<$}}node[above left=-1.5]{$3$}(lu)--node{\small\rotatebox{22.5}{$<$}}node[above left=-1.5]{$4$}(u);
 \draw[blue] (lu)+(22.5:0.2)--(d); \draw[blue] (lu)+(22.5:0.25)--(d);
 \draw[blue] (d)+(22.5:0.2)--($(u)!0.5!(ru)$); \draw[blue] (d)+(22.5:0.25)..controls(-45:0.5)..($(rd)!0.5!(r)$);
 \draw[blue] (r)+(-112.5:0.2)arc(-112.5:-247.5:0.2); \draw[blue] (rd)+(-157.5:0.2)arc(-157.5:-292.5:0.2);
 \draw[blue] (r)+(-112.5:0.25)arc(-112.5:-247.5:0.25); \draw[blue] (rd)+(-157.5:0.25)arc(-157.5:-292.5:0.25);
 \draw[blue] (u)+(-22.5:0.2)arc(-22.5:-157.5:0.2); \draw[blue] (ru)+(-67.5:0.2)arc(-67.5:-202.5:0.2); \draw[blue] (l)+(67.5:0.2)arc(67.5:-67.5:0.2); \draw[blue] (ld)+(-22.5:0.2)arc(-22.5:112.5:0.2);
 \draw[blue] (u)+(-22.5:0.25)arc(-22.5:-157.5:0.25); \draw[blue] (ru)+(-67.5:0.25)arc(-67.5:-202.5:0.25); \draw[blue] (l)+(67.5:0.25)arc(67.5:-67.5:0.25); \draw[blue] (ld)+(-22.5:0.25)arc(-22.5:112.5:0.25);
 \fill(u) circle (0.07); \fill(ru) circle (0.07); \fill(r) circle (0.07); \fill(rd) circle (0.07); \fill(d) circle (0.07); \fill(ld) circle (0.07); \fill(l) circle (0.07); \fill(lu) circle (0.07);
\end{tikzpicture}.
\]
Then $\delta_k$ contains $c_{k+1} \cdots c_{k+6}$ for $k \in \{1,2,5,6\}$, where $c_{8+h}=c_h$. Therefore, there is $k$ such that $B_{\gamma,\delta_k} \neq 0$ except for $(i,j)=(1,5)$. In which case, either $B_{\gamma,\delta_3} \neq 0$ or $B_{\gamma,\delta_7} \neq 0$ for two tagged arcs
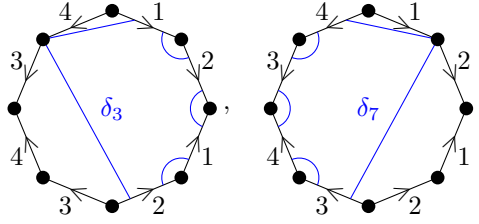
\[
\begin{tikzpicture}[baseline=0mm,scale=1.3]
 \coordinate(0)at(0,0); \node[blue] at(0) {$\delta_3$};
 \coordinate(u)at(90:1); \coordinate(d)at(-90:1); \coordinate(r)at(0:1); \coordinate(l)at(180:1);
 \coordinate(ru)at(45:1); \coordinate(rd)at(-45:1); \coordinate(lu)at(135:1); \coordinate(ld)at(-135:1);
 \coordinate(2)at(180:1.5);
 \draw(u)--node{\small\rotatebox{-22.5}{$>$}}node[above right=-1.5]{$1$}(ru)--node{\small\rotatebox{-67.5}{$>$}}node[above right=-1.5]{$2$}(r)--node{\small\rotatebox{67.5}{$>$}}node[below right=-1.5]{$1$}(rd)--node{\small\rotatebox{22.5}{$>$}}node[below right=-1.5]{$2$}(d)--node{\small\rotatebox{-22.5}{$<$}}node[below left=-1.5]{$3$}(ld)--node{\small\rotatebox{-67.5}{$<$}}node[below left=-1.5]{$4$}(l)--node{\small\rotatebox{67.5}{$<$}}node[above left=-1.5]{$3$}(lu)--node{\small\rotatebox{22.5}{$<$}}node[above left=-1.5]{$4$}(u);
 \draw[blue] (u)+(-22.5:0.25)--(lu); \draw[blue] (d)+(22.5:0.2)--(lu);
 \draw[blue] (ru)+(-67.5:0.2)arc(-67.5:-202.5:0.2); \draw[blue] (r)+(-112.5:0.2)arc(-112.5:-247.5:0.2); \draw[blue] (rd)+(-157.5:0.2)arc(-157.5:-292.5:0.2);
 \fill(u) circle (0.07); \fill (ru) circle (0.07); \fill(r)circle(0.07); \fill(rd) circle (0.07); \fill(d)circle(0.07); \fill(ld) circle (0.07); \fill(l) circle (0.07); \fill (lu) circle (0.07);
\end{tikzpicture},
   \hspace{3mm}%
\begin{tikzpicture}[baseline=0mm,scale=1.3]
 \coordinate(0)at(0,0); \node[blue] at(0) {$\delta_7$};
 \coordinate(u)at(90:1); \coordinate(d)at(-90:1); \coordinate(r)at(0:1); \coordinate(l)at(180:1);
 \coordinate(ru)at(45:1); \coordinate(rd)at(-45:1); \coordinate(lu)at(135:1); \coordinate(ld)at(-135:1);
 \coordinate(2)at(180:1.5);
 \draw(u)--node{\small\rotatebox{-22.5}{$>$}}node[above right=-1.5]{$1$}(ru)--node{\small\rotatebox{-67.5}{$>$}}node[above right=-1.5]{$2$}(r)--node{\small\rotatebox{67.5}{$>$}}node[below right=-1.5]{$1$}(rd)--node{\small\rotatebox{22.5}{$>$}}node[below right=-1.5]{$2$}(d)--node{\small\rotatebox{-22.5}{$<$}}node[below left=-1.5]{$3$}(ld)--node{\small\rotatebox{-67.5}{$<$}}node[below left=-1.5]{$4$}(l)--node{\small\rotatebox{67.5}{$<$}}node[above left=-1.5]{$3$}(lu)--node{\small\rotatebox{22.5}{$<$}}node[above left=-1.5]{$4$}(u);
 \draw[blue] (u)+(-157.5:0.25)--(ru); \draw[blue] (d)+(157.5:0.2)--(ru);
 \draw[blue] (lu)+(22.5:0.2)arc(22.5:-112.5:0.2); \draw[blue] (l)+(67.5:0.2)arc(67.5:-67.5:0.2); \draw[blue] (ld)+(112.5:0.2)arc(112.5:-22.5:0.2);
 \fill(u) circle (0.07); \fill (ru) circle (0.07); \fill(r)circle(0.07); \fill(rd) circle (0.07); \fill(d)circle(0.07); \fill(ld) circle (0.07); \fill(l) circle (0.07); \fill (lu) circle (0.07);
\end{tikzpicture}.
\]
In the same way, we can construct a desired $\delta$ for the case that the genus of $S$ is three or more.
\end{proof}

\subsection{Tagged triangulations with common intersection vectors}

A \emph{tagged triangulation} of $\cS$ is a maximal set $T$ of tagged arcs of $\cS$ such that $\Int(\gamma,\delta)=0$ for any $\gamma, \delta \in T$. We denote by $\bT_{\cS}$ the set of tagged triangulations of $\cS$. By \cite[Remark 4.2]{FoST08}, a tagged triangulation of $\cS$ decomposes $S$ into \emph{puzzle pieces} as in Figure \ref{fig:puzzle pieces}.

\begin{figure}[ht]
\begin{tikzpicture}[baseline=3mm]
 \coordinate(l)at(-150:1); \coordinate(r)at(-30:1); \coordinate(u)at(90:1);
 \draw(u)--(l)--(r)--(u);
 \fill(u)circle(0.07); \fill(l)circle(0.07); \fill(r)circle(0.07);
\end{tikzpicture}
 \hspace{5mm}%
\begin{tikzpicture}[baseline=0mm]
 \coordinate(c)at(0,0); \coordinate(u)at(0,1); \coordinate(d)at(0,-1);
 \draw(d)to[out=180,in=180](u);
 \draw(d)to[out=0,in=0](u);
 \draw(d)to[out=60,in=120,relative](c);
 \draw(d)to[out=-60,in=-120,relative]node[pos=0.8]{\rotatebox{40}{\footnotesize $\bowtie$}}(c);
 \fill(c)circle(0.07); \fill(u)circle(0.07); \fill(d)circle(0.07);
\end{tikzpicture}
 \hspace{5mm}%
\begin{tikzpicture}[baseline=0mm]
 \coordinate(l)at(-0.5,0); \coordinate(r)at(0.5,0); \coordinate(d)at(0,-1);
 \draw(0,0)circle(1);
 \draw(d)to[out=170,in=-130](l);
 \draw(d)to[out=95,in=0]node[pos=0.8]{\rotatebox{60}{\footnotesize $\bowtie$}}(l);
 \draw(d)to[out=85,in=180](r);
 \draw(d)to[out=10,in=-50] node[pos=0.8]{\rotatebox{10}{\footnotesize $\bowtie$}}(r);
 \fill(l)circle(0.07); \fill(r)circle(0.07); \fill(d)circle(0.07);
\end{tikzpicture}
   \hspace{5mm}%
\begin{tikzpicture}[baseline=3mm]
 \coordinate(c)at(0,0); \coordinate(u)at(90:1); \coordinate(r)at(-30:1); \coordinate(l)at(210:1);
 \draw(c)to[out=60,in=120,relative](u);
 \draw(c)to[out=-60,in=-120,relative]node[pos=0.8]{\rotatebox{40}{\footnotesize $\bowtie$}}(u);
 \draw(c)to[out=60,in=120,relative](r);
 \draw(c)to[out=-60,in=-120,relative]node[pos=0.8]{\rotatebox{-80}{\footnotesize $\bowtie$}}(r);
 \draw(c)to[out=60,in=120,relative](l);
 \draw(c)to[out=-60,in=-120,relative] node[pos=0.8]{\rotatebox{160}{\footnotesize $\bowtie$}}(l);
 \fill(c)circle(0.07); \fill(u)circle(0.07); \fill(l)circle(0.07); \fill(r)circle(0.07);
\end{tikzpicture}
\caption{A complete list of puzzle pieces of $\cS$ up to changing all tags at each puncture, where the rightmost puzzle piece forms a tagged triangulation of a sphere with exactly four punctures}
\label{fig:puzzle pieces}
\end{figure}
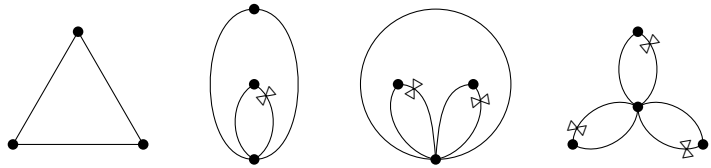

Let $T \in \bT_{\cS}$. For $\gamma \in T$, there is a unique tagged arc $\gamma' \notin T$ such that $\mu_{\gamma}T:=(T\setminus\{\gamma\})\sqcup\{\gamma'\}\in\bT_{\cS}$. We call $\mu_{\gamma}T$ the \emph{flip} of $T$ at $\gamma$. We denote by $\bT_{\cS}^T$ the set of tagged triangulations of $\cS$ obtained from $T$ by sequences of flips, and by $\bA_{\cS}^T$ the set of tagged arcs of tagged triangulations contained in $\bT_{\cS}^T$. It is elementary that flips are compatible with tagged rotations, that is, we have $\rho(\mu_{\gamma}T)=\mu_{\rho(\gamma)}\rho(T)$. It implies that $\rho: \bT_{\cS}^T \rightarrow \bT_{\cS}^{\rho (T)}$ is a bijection which commutes with flips. In almost cases, its domain and codomain coincide as follows.

\begin{theorem}[{\cite[Theorem 7.9 and Proposition 7.10]{FoST08}}]\label{thm:transitivity of flips}
 If $\cS$ is not a closed surface with exactly one puncture, then we have $\bT_{\cS}^T=\bT_{\cS}$, that is, any two tagged triangulations of $\cS$ are connected by a sequence of flips. If $\cS$ is a closed surface with exactly one puncture, then we have $\bT_{\cS}^T\sqcup\bT_{\cS}^{\rho(T)}=\bT_{\cS}$, in particular, all tags appearing in $\bT_{\cS}^T$ are the same.
\end{theorem}

For $\gamma\in\bA_{\cS}$, we set two kinds of \emph{intersection vectors}
\[
(T | \gamma):=((t | \gamma))_{t \in T} \in \bZ^{T} \text{ and } \Int(T,\gamma):=(\Int(t,\gamma))_{t \in T} \in \bZ^{T}. 
\]
Note that the former induces the denominator vector of the corresponding cluster variable, while the latter induces its $f$-vector and the dimension vector of the corresponding indecomposable module (see Theorem \ref{thm:bijection-x} and under Theorem \ref{thm:bijection-sC}).

\begin{definition}\label{def:CIV}
Let $T\in\bT_{\cS}$. We say that
\begin{enumerate}
\item $T$ has \emph{common intersection vectors} (\emph{CIV}, for short) if $(T | \gamma)=\Int(T,\gamma)$ for any $\gamma\in\bA_{\cS} \setminus T$;
\item $T$ has \emph{weak common intersection vectors} (\emph{wCIV}, for short) if $(T | \gamma)=\Int(T,\gamma)$ for any $\gamma\in\bA_{\cS}^T \setminus T$.
\end{enumerate}
\end{definition}

Theorem \ref{thm:transitivity of flips} implies that CIV and wCIV are equivalent if $\cS$ is not a closed surface with exactly one puncture. In which case, we get the following.

\begin{proposition}\label{prop:once-punctured CIV}
Suppose that $\cS$ is a closed surface with exactly one puncture. Then all $T \in \bT_{\cS}$ do not have CIV. Moreover, the following are equivalent:
\begin{enumerate}
\item There is at least one $T\in\bT_{\cS}$ with wCIV;
\item All $T\in\bT_{\cS}$ have wCIV;
\item $S$ is a torus.
\end{enumerate}
\end{proposition}

\begin{proof}
Using the notations in Proposition \ref{prop:difference between intersection numbers}, the first claim follows from $C''_{\gamma,\rho(\gamma)}\neq 0$ for any $\gamma\in\bA_{\cS}$ under the assumption. Let $T \in \bT_{\cS}$ and $\gamma\in T$. We have $C''_{\gamma,\delta}=D''_{\gamma,\delta}=0$ for any $\delta \in \bA_{\cS}^T\setminus\{\gamma\}$ since all tags of $\gamma$ and $\delta$ are the same by Theorem \ref{thm:transitivity of flips}. By Lemma \ref{lem:once-punctured cases}, there is a $\delta \in \bA_{\cS}$ such that $B_{\gamma,\delta}\neq 0$ if and only if $S$ is not a torus. In particular, we can take $\delta$ as in $\bA_{\cS}^T$ since $B_{\gamma,\delta}$ is independent of tags. Therefore, $T$ have wCIV if and only if $S$ is a torus. This gives the desired equivalences.
\end{proof}

We are ready to give the main result of this section.

\begin{theorem}\label{thm:CIV}
For $T\in\bT_{\cS}$, the following are equivalent:
\begin{enumerate}
\item $T$ has CIV;
\item $T$ has neither loops nor tagged arcs connecting punctures;
\item $\Int(\gamma,\rho(\gamma))=1$ for any $\gamma \in T$;
\item $\Int(\gamma,\rho^{-1}(\gamma))=1$ for any $\gamma \in T$.
\end{enumerate}
\end{theorem}

\begin{proof}
Using the notations in Proposition \ref{prop:difference between intersection numbers}, Lemma \ref{lem:observations}(1) implies that $T$ have CIV if and only if $B_{\gamma,\delta}=C''_{\gamma,\delta}=D''_{\gamma,\delta}=0$ for any $\gamma \in T$ and $\delta\in\bA_{\cS}\setminus\{T\}$. By Lemma \ref{lem:observations}(2) and (3), their equalities hold for $T$ satisfying (2), that is, (2) implies (1). The inverse follows from Lemmas \ref{lem:observations}(3) and \ref{lem:loop B}. It is straightforward to check the remaining equivalences. In fact, Table \ref{table:Int(g,rho g)} lists all kinds of tagged arcs up to changing all tags at each puncture. It gives that $\gamma\in\bA_{\cS}$ is neither a loop nor a tagged arc connecting punctures if and only if $\Int(\gamma,\rho(\gamma))=1$. Similarly, it is equivalent to $\Int(\gamma,\rho^{-1}(\gamma))=1$. Thus (2), (3) and (4) are equivalent.
\end{proof}

\renewcommand{\arraystretch}{2}
{\begin{table}[ht]
\begin{tabular}{c|c|c|c|c|c}
&
\begin{tikzpicture}[baseline=0mm]
 \coordinate(u)at(0,1); \coordinate(d)at(0,-1); \coordinate(ru)at(0.5,1); \coordinate(ld)at(-0.5,-1);
 \fill[pattern=north east lines](-1,1)--(-1,1.3)--(1,1.3)--(1,1); \fill[pattern=north east lines](-1,-1)--(-1,-1.3)--(1,-1.3)--(1,-1);
 \draw(-1,1)--(1,1) (-1,-1)--(1,-1);
 \draw(u)--node[right,pos=0.7]{$\gamma$}(d); \draw[blue](ru)--node[left,pos=0.7]{$\rho(\gamma)$}(ld);
 \fill(u)circle(0.07); \fill(d)circle(0.07); \fill(ru)circle(0.07); \fill(ld)circle(0.07);
 \node at(0,-1.5){};
\end{tikzpicture}
&
\begin{tikzpicture}[baseline=0mm]
 \coordinate(u)at(0,1); \coordinate(d)at(0,-1); \coordinate(ld)at(-0.5,-1);
 \fill[pattern=north east lines](-1,-1)--(-1,-1.3)--(1,-1.3)--(1,-1);
 \draw(-1,-1)--(1,-1);
 \draw(u)--node[right]{$\gamma$}(d);
 \draw[blue](u)to[out=-120,in=90]node[left]{$\rho(\gamma)$}node[pos=0.15]{\rotatebox{-30}{\footnotesize $\bowtie$}}(ld);
 \fill(u)circle(0.07); \fill(d)circle(0.07); \fill(ld)circle(0.07);
\end{tikzpicture}
&
\begin{tikzpicture}[baseline=0mm]
 \coordinate(u)at(0,1); \coordinate(d)at(0,-1);
 \draw(u)--node[right]{$\gamma$}(d);
 \draw[blue](u)to[out=-130,in=130]node[left,pos=0.5]{$\rho(\gamma)$}node[pos=0.1]{\rotatebox{-40}{\footnotesize $\bowtie$}}node[pos=0.9]{\rotatebox{40}{\footnotesize $\bowtie$}}(d);
 \fill(u)circle(0.07); \fill(d)circle(0.07);
\end{tikzpicture}
&
\begin{tikzpicture}[baseline=0mm]
 \coordinate(u)at(0,1); \coordinate(d)at(0,-1); \coordinate(ld)at(-0.5,-1);
 \fill[pattern=north east lines](-1,-1)--(-1,-1.3)--(1,-1.3)--(1,-1); \draw(-1,-1)--(1,-1);
 \draw(d)..controls(-0.7,-0.5)and(-0.7,0.7)..(0,0.7); \draw(d)..controls(0.7,-0.5)and(0.7,0.7)..node[right,pos=0.3]{$\gamma$}(0,0.7);
 \draw[blue](ld)..controls(-1,0)and(-0.7,1)..node[left,pos=0.7]{$\rho(\gamma)$}(u); \draw[blue](ld)..controls(1,0)and(0.7,1)..(u);
 \fill(d)circle(0.07); \fill(ld)circle(0.07);
 \draw[dotted](0,0.2)circle(0.3);
\end{tikzpicture}
&
\begin{tikzpicture}[baseline=1mm]
 \coordinate(u)at(0,1); \coordinate(d)at(0,-1);
 \draw(d)..controls(-0.7,-0.3)and(-0.7,0.7)..(0,0.7); \draw(d)..controls(0.7,-0.3)and(0.7,0.7)..node[left,pos=0.2]{$\gamma$}(0,0.7);
 \draw[blue](d)..controls(-1,-1)and(-1,1)..node[pos=0.1]{\rotatebox{70}{\footnotesize $\bowtie$}}(u)node[above]{$\rho(\gamma)$};
 \draw[blue](d)..controls(1,-1)and(1,1)..node[pos=0.1]{\rotatebox{110}{\footnotesize $\bowtie$}}(u);
 \fill(d)circle(0.07);
 \draw[dotted](0,0.2)circle(0.3);
\end{tikzpicture}
\\\hline
$\Int(\gamma,\rho(\gamma))$ & 1 & 1 & 2 & 2 & 4
\end{tabular}\vspace{3mm}
\caption{QZ-intersection number of each tagged arc $\gamma$ and its tagged rotation $\rho(\gamma)$}
\label{table:Int(g,rho g)}
\end{table}}

\begin{theorem}\label{thm:wCIV}
For $T \in \bT_{\cS}$, it has wCIV if and only if it has CIV or $\cS$ is a torus with exactly one puncture.
\end{theorem}

\begin{proof}
Since CIV and wCIV are equivalent except that $\cS$ is a closed surface with exactly one puncture, the assertion follows from Proposition \ref{prop:once-punctured CIV}.
\end{proof}

In general, the conditions (1) and (2) in Proposition \ref{prop:once-punctured CIV} are not equivalent. We also characterize $\cS$ satisfying each condition.

\begin{theorem}\label{thm:CIV cS}
Let $\cS$ be a marked surface.
\begin{enumerate}
\item There is at least one $T\in\bT_{\cS}$ with CIV if and only if $\cS$ is one of the following:
\begin{itemize}
 \item a marked surface of genus zero with $\partial S\neq\emptyset$;
 \item a marked surface of positive genus with $\partial S\neq\emptyset$, $|M|\ge3$ and $|M\cap\partial S|\ge2$.
\end{itemize}
\item All $T\in\bT_{\cS}$ have CIV if and only if $\cS$ is one of the following:
\begin{itemize}
 \item a polygon with at most one puncture;
 \item an annulus with $|M|=2$.
\end{itemize}
\end{enumerate}
In (1) and (2), we also have the similar statements obtained by replacing CIV with wCIV and adding the case that $\cS$ is a torus with exactly one puncture.
\end{theorem}

\begin{proof}[Proof of Theorem \ref{thm:CIV cS}(2)]
By Theorem \ref{thm:CIV}, all $T\in\bT_{\cS}$ have CIV if and only if there are neither loops nor tagged arcs connecting punctures in $\cS$. It is easy to check that $\cS$ appearing in (2) satisfies such conditions. In other cases, $\cS$ satisfies at least one of the following:
\begin{itemize}
 \item there are at least two punctures;
 \item the genus of $S$ is positive;
 \item there are at least two boundary components and $|M| \ge 3$.
\end{itemize}
In each case, we can take a loop or a tagged arc connecting punctures as in Figure \ref{fig:cSCIV2 proof}. Thus (2) holds.
\end{proof}

\begin{figure}[htp]
\begin{tikzpicture}[baseline=0mm,scale=1]
 \coordinate(l)at(-0.5,0); \coordinate(r)at(0.5,0);
 \draw[blue] (l)--(r);
 \fill(l) circle (0.07); \fill(r) circle (0.07);
\end{tikzpicture}
   \hspace{3mm}
\begin{tikzpicture}[baseline=0mm]
 \coordinate(p)at(0,-0.4);
 \draw(-1.3,0)to[out=90,in=90](1.3,0); \draw(-1.3,0)to[out=-90,in=-90](1.3,0);
 \draw(-0.5,0.1)to[out=-30,in=-150](0.5,0.1); \draw(-0.3,0)to[out=30,in=150](0.3,0);
 \draw[blue](p)..controls(1.3,-0.3)and(1.3,0.5)..(0,0.5);
 \draw[blue](p)..controls(-1.3,-0.3)and(-1.3,0.5)..(0,0.5);
 \fill(p)circle(0.07);
\end{tikzpicture}
   \hspace{3mm}
\begin{tikzpicture}[baseline=0mm]
 \coordinate(p)at(0,-0.3); \draw[pattern=north east lines](0,-0.5)circle (2mm);
 \draw(-1.3,0)to[out=90,in=90](1.3,0); \draw(-1.3,0)to[out=-90,in=-90](1.3,0);
 \draw(-0.5,0.1)to[out=-30,in=-150](0.5,0.1); \draw(-0.3,0)to[out=30,in=150](0.3,0);
 \draw[blue](p)..controls(1,-0.2)and(1,0.5)..(0,0.5);
 \draw[blue](p)..controls(-1,-0.2)and(-1,0.5)..(0,0.5);
 \fill(p)circle(0.07);
\end{tikzpicture}
   \hspace{3mm}
\begin{tikzpicture}[baseline=0mm]
 \coordinate(0)at(0,0); \coordinate(u)at(0,1);
 \coordinate(cu)at(0,0.2); \coordinate(p)at(0,-0.6);
 \draw[pattern=north east lines] (0) circle (2mm); \draw(0) circle (10mm);
 \draw[blue] (p) .. controls (-0.7,0) and (-0.4,0.5) .. (0,0.5);
 \draw[blue] (p) .. controls (0.7,0) and (0.4,0.5) .. (0,0.5);
 \fill(u)circle(0.07); \fill (cu) circle (0.07); \fill(p)circle(0.07);
\end{tikzpicture}
   \hspace{3mm}
\begin{tikzpicture}[baseline=0mm]
 \coordinate(0)at(0,0); \coordinate(u)at(0,1);
 \coordinate(cu)at(0,0.2); \coordinate(p)at(0,-1);
 \draw[pattern=north east lines] (0) circle (2mm); \draw(0) circle (10mm);
 \draw[blue] (p) .. controls (-0.8,0) and (-0.4,0.5) .. (0,0.5);
 \draw[blue] (p) .. controls (0.8,0) and (0.4,0.5) .. (0,0.5);
 \fill(u)circle(0.07); \fill (cu) circle (0.07); \fill(p)circle(0.07);
\end{tikzpicture}
 \caption{A loop or a tagged arc connecting punctures on each marked surface which does not appear in Theorem \ref{thm:CIV cS}(2)}
\label{fig:cSCIV2 proof}
\end{figure}
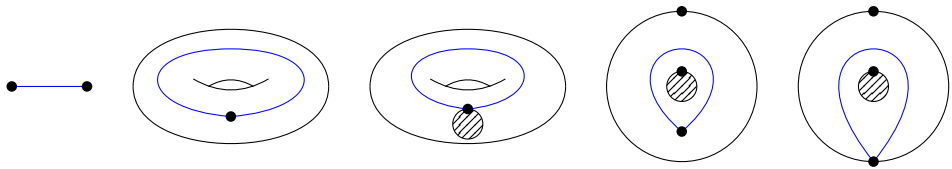

To prove Theorem \ref{thm:CIV cS}(1), we prepare the following property.

\begin{proposition}\label{prop:adding}
Suppose that $\partial S\neq\emptyset$ and there is a $T\in\bT_{\cS}$ with CIV. Then there is a $T'\in\bT_{\cS'}$ with CIV if $\cS'$ is a marked surface obtained from $\cS$ by one of the following:
\begin{enumerate}
\item adding a marked point on $\partial S$;
\item adding a puncture;
\item adding a boundary component with exactly one marked point.
\end{enumerate}
\end{proposition}

\begin{proof}
Up to homotopy, we can take a puzzle piece $\triangle$ of $T$ whose side is a boundary segment of $\cS$ such that the added marked point or boundary component is in $\triangle$. We obtain $T'\in\bT_{\cS'}$ from $T$ by adding puzzle pieces in each case as follows:
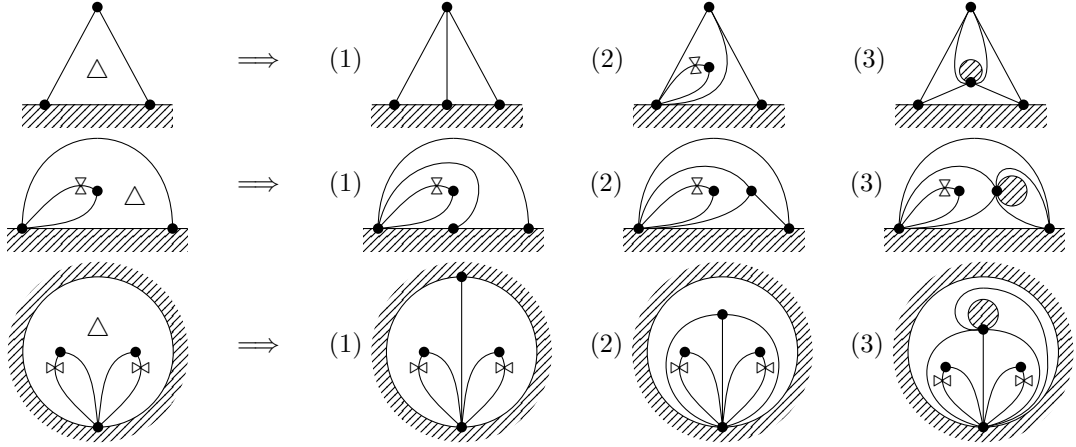
\begin{eqnarray*}
\begin{tikzpicture}[baseline=5mm]
 \coordinate(u)at(0,1.3);
 \coordinate(l)at(-0.7,0); \coordinate(r)at(0.7,0);
 \fill[pattern=north east lines] (-1,0)--(-1,-0.3)--(1,-0.3)--(1,0);
 \draw(l)--(u)--(r); \draw(-1,0)--(1,0);
 \fill(u)circle(0.07); \fill(l)circle(0.07); \fill(r)circle(0.07);
 \node at(0,0.45) {$\triangle$};
\end{tikzpicture}
 \hspace{5mm}&\Longrightarrow&\hspace{3mm}
 (1)\ 
\begin{tikzpicture}[baseline=5mm]
 \coordinate(u)at(0,1.3);
 \coordinate(l)at(-0.7,0); \coordinate(r)at(0.7,0); \coordinate(0)at(0,0);
 \fill[pattern=north east lines] (-1,0)--(-1,-0.3)--(1,-0.3)--(1,0);
 \draw(l)--(u)--(r) (u)--(0); \draw(-1,0)--(1,0);
 \fill(u)circle(0.07); \fill(l)circle(0.07); \fill(r)circle(0.07); \fill(0)circle(0.07);
\end{tikzpicture}
 \hspace{9mm}
 (2)\ 
\begin{tikzpicture}[baseline=5mm]
 \coordinate(u)at(0,1.3);
 \coordinate(l)at(-0.7,0); \coordinate(r)at(0.7,0); \coordinate(p)at(0,0.5);
 \fill[pattern=north east lines] (-1,0)--(-1,-0.3)--(1,-0.3)--(1,0);
 \draw(l)--(u)--(r); \draw(-1,0)--(1,0);
 \draw(l)to[out=20,in=120,relative] node[pos=0.8]{\rotatebox{95}{\footnotesize $\bowtie$}} (p);
 \draw(l)to[out=-20,in=-120,relative] (p);
 \draw(l)..controls(0.2,0) and (0.5,0.4)..(u);
 \fill(u)circle(0.07); \fill(l)circle(0.07); \fill(r)circle(0.07); \fill(p)circle(0.07);
\end{tikzpicture}
 \hspace{9mm}
 (3)\ 
\begin{tikzpicture}[baseline=5mm]
 \coordinate(u)at(0,1.3);
 \coordinate(l)at(-0.7,0); \coordinate(r)at(0.7,0); \coordinate(p)at(0,0.3);
 \draw[pattern=north east lines] (0,0.45) circle (1.5mm);
 \fill[pattern=north east lines] (-1,0)--(-1,-0.3)--(1,-0.3)--(1,0);
 \draw(l)--(u)--(r) (l)--(p)--(r); \draw(-1,0)--(1,0);
 \draw(u) to [out=-30,in=-70,relative] (p); \draw(u) to [out=30,in=70,relative] (p);
 \fill(u)circle(0.07); \fill(l)circle(0.07); \fill(r)circle(0.07); \fill(p)circle(0.07);
\end{tikzpicture}
\\
\begin{tikzpicture}[baseline=5mm]
 \coordinate(u)at(0,1.2); \coordinate(p)at(0,0.5);
 \coordinate(l)at(-1,0); \coordinate(r)at(1,0);
 \fill[pattern=north east lines] (-1.2,0)--(-1.2,-0.3)--(1.2,-0.3)--(1.2,0);
 \draw(l) ..controls(-1,0.8)and(-0.5,1.2)..(u); \draw(r) ..controls(1,0.8)and(0.5,1.2)..(u);
 \draw(l)to[out=20,in=120,relative] node[pos=0.8]{\rotatebox{90}{\footnotesize $\bowtie$}} (p);
 \draw(l)to[out=-20,in=-120,relative] (p);
 \draw(-1.2,0)--(1.2,0);
 \fill(p)circle(0.07); \fill(l)circle(0.07); \fill(r)circle(0.07);
 \node at(0.5,0.4) {$\triangle$};
\end{tikzpicture}
 \hspace{3mm}&\Longrightarrow&\hspace{3mm}
 (1)
\begin{tikzpicture}[baseline=5mm]
 \coordinate(u)at(0,1.2); \coordinate(p)at(0,0.5); \coordinate(d)at(0,0);
 \coordinate(l)at(-1,0); \coordinate(r)at(1,0);
 \fill[pattern=north east lines] (-1.2,0)--(-1.2,-0.3)--(1.2,-0.3)--(1.2,0);
 \draw(l) ..controls(-1,0.8)and(-0.5,1.2)..(u); \draw(r) ..controls(1,0.8)and(0.5,1.2)..(u);
 \draw(l)to[out=20,in=120,relative] node[pos=0.8]{\rotatebox{90}{\footnotesize $\bowtie$}} (p);
 \draw(l)to[out=-20,in=-120,relative] (p);
 \draw(-1.2,0)--(1.2,0);
 \draw(l) ..controls(-0.8,1)and(0.2,1)..(0.3,0.7); \draw(d) ..controls(0.3,0.1)and(0.4,0.5)..(0.3,0.7);
 \fill(p)circle(0.07); \fill(l)circle(0.07); \fill(r)circle(0.07); \fill(d)circle(0.07);
\end{tikzpicture}
 \hspace{6mm}
 (2)
\begin{tikzpicture}[baseline=5mm]
 \coordinate(u)at(0,1.2); \coordinate(p)at(0,0.5); \coordinate(d)at(0.5,0.5);
 \coordinate(l)at(-1,0); \coordinate(r)at(1,0);
 \fill[pattern=north east lines] (-1.2,0)--(-1.2,-0.3)--(1.2,-0.3)--(1.2,0);
 \draw(l) ..controls(-1,0.8)and(-0.5,1.2)..(u); \draw(r) ..controls(1,0.8)and(0.5,1.2)..(u);
 \draw(l)to[out=20,in=120,relative] node[pos=0.8]{\rotatebox{90}{\footnotesize $\bowtie$}} (p);
 \draw(l)to[out=-20,in=-120,relative] (p);
 \draw(-1.2,0)--(1.2,0); \draw(r)--(d);
 \draw(l) ..controls(-0.8,1)and(0.2,1)..(d); \draw(l) ..controls(-0.2,0)and(0.4,0.2)..(d);
 \fill(p)circle(0.07); \fill(l)circle(0.07); \fill(r)circle(0.07); \fill(d)circle(0.07);
\end{tikzpicture}
 \hspace{6mm}
 (3)
\begin{tikzpicture}[baseline=5mm]
 \coordinate(u)at(0,1.2); \coordinate(p)at(-0.2,0.5); \coordinate(d)at(0.3,0.5);
 \coordinate(l)at(-1,0); \coordinate(r)at(1,0);
 \fill[pattern=north east lines] (-1.2,0)--(-1.2,-0.3)--(1.2,-0.3)--(1.2,0);
 \draw[pattern=north east lines] (0.5,0.5) circle (0.2);
 \draw(l) ..controls(-1,0.8)and(-0.5,1.2)..(u); \draw(r) ..controls(1,0.8)and(0.5,1.2)..(u);
 \draw(l)to[out=20,in=120,relative] node[pos=0.8]{\rotatebox{90}{\footnotesize $\bowtie$}} (p);
 \draw(l)to[out=-20,in=-120,relative] (p);
 \draw(-1.2,0)--(1.2,0);
 \draw(l) ..controls(-0.8,1)and(0.1,1)..(d); \draw(l) ..controls(-0.2,0)and(0.2,0.2)..(d);
 \draw(r) ..controls(0.7,0)and(0.2,0.3)..(d); \draw(r) ..controls(0.9,0.9)and(0.2,1)..(d);
 \fill(p)circle(0.07); \fill(l)circle(0.07); \fill(r)circle(0.07); \fill(d)circle(0.07);
\end{tikzpicture}
\\
\begin{tikzpicture}[baseline=0mm]
 \coordinate(0)at(0,0); \coordinate(l)at(-0.5,0); \coordinate(r)at(0.5,0); \coordinate(d)at(0,-1);
 \draw(0) circle (1);
 \draw(d)to[out=-20,in=-120,relative](l);
 \draw(d)to[out=20,in=120,relative] node[pos=0.8]{\rotatebox{0}{\footnotesize $\bowtie$}}(l);
 \draw(d)to[out=-20,in=-120,relative] node[pos=0.8]{\rotatebox{0}{\footnotesize $\bowtie$}}(r);
 \draw(d)to[out=20,in=120,relative](r);
 \fill[pattern=north east lines] (-1.2,0)arc(180:-180:12mm)--(-1,0)arc(-180:180:10mm);
 \fill(l)circle(0.07); \fill(r)circle(0.07); \fill(d)circle(0.07);
 \node at(0,0.3) {$\triangle$};
\end{tikzpicture}
 \hspace{3mm}&\Longrightarrow&\hspace{3mm}
 (1)\ 
\begin{tikzpicture}[baseline=0mm]
 \coordinate(0)at(0,0); \coordinate(l)at(-0.5,0); \coordinate(r)at(0.5,0); \coordinate(d)at(0,-1);
 \coordinate(u)at(0,1);
 \draw(0) circle (1); \draw(u)--(d);
 \draw(d)to[out=-20,in=-120,relative](l);
 \draw(d)to[out=20,in=120,relative] node[pos=0.8]{\rotatebox{0}{\footnotesize $\bowtie$}}(l);
 \draw(d)to[out=-20,in=-120,relative] node[pos=0.8]{\rotatebox{0}{\footnotesize $\bowtie$}}(r);
 \draw(d)to[out=20,in=120,relative](r);
 \fill[pattern=north east lines] (-1.2,0)arc(180:-180:12mm)--(-1,0)arc(-180:180:10mm);
 \fill(l)circle(0.07); \fill(r)circle(0.07); \fill(d)circle(0.07); \fill(u)circle(0.07);
\end{tikzpicture}
 \hspace{5mm}
 (2)\ 
\begin{tikzpicture}[baseline=0mm]
 \coordinate(0)at(0,0); \coordinate(l)at(-0.5,0); \coordinate(r)at(0.5,0); \coordinate(d)at(0,-1);
 \coordinate(u)at(0,0.5);
 \draw(0) circle (1); \draw(u)--(d);
 \draw(d)to[out=-20,in=-120,relative](l);
 \draw(d)to[out=20,in=120,relative] node[pos=0.8]{\rotatebox{0}{\footnotesize $\bowtie$}}(l);
 \draw(d)to[out=-20,in=-120,relative] node[pos=0.8]{\rotatebox{0}{\footnotesize $\bowtie$}}(r);
 \draw(d)to[out=20,in=120,relative](r);
 \draw(d) ..controls(-1,-0.8)and(-1,0.5)..(u);
 \draw(d) ..controls(1,-0.8)and(1,0.5)..(u);
 \fill[pattern=north east lines] (-1.2,0)arc(180:-180:12mm)--(-1,0)arc(-180:180:10mm);
 \fill(l)circle(0.07); \fill(r)circle(0.07); \fill(d)circle(0.07); \fill(u)circle(0.07);
\end{tikzpicture}
 \hspace{5mm}
 (3)\ 
\begin{tikzpicture}[baseline=0mm]
 \coordinate(0)at(0,0); \coordinate(l)at(-0.5,-0.2); \coordinate(r)at(0.5,-0.2); \coordinate(d)at(0,-1);
 \coordinate(u)at(0,0.3);
 \draw(0) circle (1); \draw[pattern=north east lines] (0,0.5) circle (0.2); \draw(u)--(d);
 \draw(d)to[out=-20,in=-120,relative](l);
 \draw(d)to[out=20,in=120,relative] node[pos=0.8]{\rotatebox{0}{\footnotesize $\bowtie$}}(l);
 \draw(d)to[out=-20,in=-120,relative] node[pos=0.8]{\rotatebox{0}{\footnotesize $\bowtie$}}(r);
 \draw(d)to[out=20,in=120,relative](r);
 \draw(d) ..controls(-1,-0.8)and(-1,0.3)..(u);
 \draw(d) ..controls(1,-0.8)and(1,0.3)..(u);
 \draw(d) ..controls(1,-0.85)and(0.85,-0.1)..(0.85,0);
 \draw(0.85,0) ..controls(0.8,0.7)and(0.3,0.85)..(0,0.85);
 \draw(0,0.85) ..controls(-0.5,0.85)and(-0.5,0.2)..(u);
 \fill[pattern=north east lines] (-1.2,0)arc(180:-180:12mm)--(-1,0)arc(-180:180:10mm);
 \fill(l)circle(0.07); \fill(r)circle(0.07); \fill(d)circle(0.07); \fill(u)circle(0.07);
\end{tikzpicture}
\end{eqnarray*}
 In all cases, $T'$ has neither loops nor tagged arcs connecting punctures since $T$ is so. Thus it also has CIV by Theorem \ref{thm:CIV}.
\end{proof}

\begin{proof}[Proof of Theorem \ref{thm:CIV cS}(1)]
If $\partial S=\emptyset$, then any tagged arc of $\cS$ connects punctures. Thus any $T \in \bT_{\cS}$ does not have CIV by Theorem \ref{thm:CIV}. Therefore, we assume that $\partial S \neq \emptyset$. Let $g$ be the genus of $S$ and $P$ the set of punctures of $\cS$. We consider five cases.\vspace{2mm}

\noindent\underline{Case 1. $g=0$.}
By Theorem \ref{thm:CIV cS}(2) and Proposition \ref{prop:adding}, there is a $T\in\bT_{\cS}$ with CIV except for a monogon with at least two punctures since we assume that $\cS$ is not a monogon with at most one puncture. For a monogon with exactly two punctures, two pairs of conjugate arcs form a tagged triangulation with CIV as in the proof of Proposition \ref{prop:adding}. By Proposition \ref{prop:adding} again, there is a $T\in\bT_{\cS}$ with CIV for any $\cS$ with $g=0$.\vspace{2mm}

\noindent\underline{Case 2. $g \ge 1, |M \setminus P|=1$.}
Suppose that there is a $T\in\bT_{\cS}$ with CIV, that is, $T$ has neither loops nor tagged arcs connecting punctures by Theorem \ref{thm:CIV}. By this condition, $T$ must be constructed as follows:
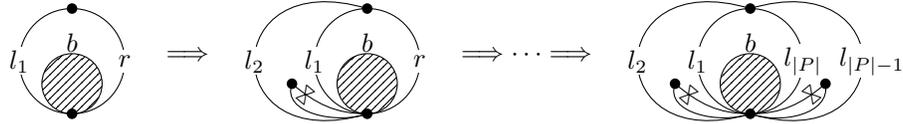
\vspace{-3mm}
\[
\begin{tikzpicture}[baseline=0mm]
 \coordinate(0)at(0,0); \coordinate(u) at(0,0.7); \coordinate(d) at(0,-0.7);
 \draw(0)circle(7mm);
 \draw[pattern=north east lines] (0,-0.3) circle (4mm);
 \fill(u)circle(0.07); \fill(d)circle(0.07);
 \node at(0,0.25){$b$}; \node[fill=white,inner sep=1]at(-0.7,0){$l_1$}; \node[fill=white,inner sep=2]at(0.7,0){$r$};
\end{tikzpicture}
\hspace{3mm}\Longrightarrow
\begin{tikzpicture}[baseline=0mm]
 \coordinate(0)at(0,0); \coordinate(u)at(0,0.7); \coordinate(d)at(0,-0.7); \coordinate(p)at(-1,-0.3);
 \draw(0)circle(7mm);
 \draw[pattern=north east lines] (0,-0.3) circle (4mm);
 \draw(d)..controls(-2,-1.3)and(-2,1.3)..node[fill=white,inner sep=1]{$l_2$}(u);
 \draw(d)to[out=200,in=-110](p);
 \draw(d)to[out=180,in=-40]node[pos=0.8]{\rotatebox{50}{\footnotesize $\bowtie$}}(p);
 \fill(u)circle(0.07); \fill(d)circle(0.07); \fill(p)circle(0.07);
 \node at(0,0.25){$b$}; \node[fill=white,inner sep=1]at(-0.7,0){$l_1$}; \node[fill=white,inner sep=2]at(0.7,0){$r$};
\end{tikzpicture}
\hspace{3mm}\Longrightarrow\cdots\Longrightarrow
\begin{tikzpicture}[baseline=0mm]
 \coordinate(0)at(0,0); \coordinate(u)at(0,0.7); \coordinate(d)at(0,-0.7); \coordinate(p)at(-1,-0.3); \coordinate(q)at(1,-0.3);
 \draw(0)circle(7mm);
 \draw[pattern=north east lines] (0,-0.3) circle (4mm);
 \draw(d)..controls(-2,-1.3)and(-2,1.3)..node[fill=white,inner sep=1]{$l_2$}(u);
 \draw(d)to[out=200,in=-110](p);
 \draw(d)to[out=180,in=-40]node[pos=0.8]{\rotatebox{50}{\footnotesize $\bowtie$}}(p);
 \draw(d)..controls(2,-1.3)and(2,1.3)..node[fill=white,inner sep=1]{\ \ $l_{|P|-1}$}(u);
 \draw(d)to[out=-20,in=-80](q);
 \draw(d)to[out=0,in=-140]node[pos=0.8]{\rotatebox{-50}{\footnotesize $\bowtie$}}(q);
 \fill(u)circle(0.07); \fill(d)circle(0.07); \fill(p)circle(0.07); \fill(q)circle(0.07);
 \node at(0,0.25){$b$}; \node[fill=white,inner sep=1]at(-0.7,0){$l_1$}; \node[fill=white,inner sep=1]at(0.7,0){$l_{|P|}$};
\end{tikzpicture}
\vspace{-3mm}\]
In fact, for the unique boundary segment $b$, a puzzle piece of $T$ with $b$ must be a triangle $brl_1$ whose one vertex is a puncture. If $l_1=r$, then $S$ is a sphere, a contradiction. Suppose that $l_1\neq r$. Then another puzzle pieces of $T$ with $l_1$ must be a digon $l_1l_2$. Similarly, a new puzzle piece of $T$ with $l_i$ must be a digon $l_il_{i+1}$ for $1 \le i \le |P|-1$, and $l_{|P|}$ must be $r$. Then $\cS$ is a sphere, a contradiction. Therefore, there is no $T\in\bT_{\cS}$ with CIV.\vspace{2mm}

\noindent\underline{Case 3. $g \ge 1, |M|=2, P=\emptyset$.}
Suppose that there is a $T\in\bT_{\cS}$ with CIV. By $P=\emptyset$, all puzzle pieces on $\cS$ are triangles. Moreover, $\cS$ has one of the following:
\begin{itemize}
 \item[(3a)] one boundary component with two marked points;
 \item[(3b)] two boundary components with one marked point.
\end{itemize}
In the case (3a), a triangle of $T$ with a boundary segment must contain a loop, a contradiction. In the case (3b), there is at least one $\gamma\in\bA_{\cS}$ connecting the distinct marked points. Then each triangle of $T$ with $\gamma$ must contain a boundary segment as follows:
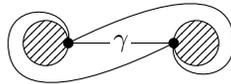
\[
\begin{tikzpicture}[baseline=0mm]
 \coordinate(l) at(-0.7,0); \coordinate(r) at(0.7,0);
 \draw[pattern=north east lines] (-1,0) circle (3mm); \draw[pattern=north east lines] (1,0) circle (3mm);
 \draw(l)--node[fill=white,inner sep=1]{$\gamma$}(r);
 \draw(l) ..controls(-0.6,0.5)and(-1.5,0.6)..(-1.5,0); \draw(r) ..controls(-0.7,-0.7)and(-1.5,-0.7)..(-1.5,0);
 \draw(r) ..controls(0.6,-0.5)and(1.5,-0.6)..(1.5,0); \draw(l) ..controls(0.7,0.7)and(1.5,0.7)..(1.5,0);
 \fill(l)circle(0.07); \fill(r)circle(0.07);
\end{tikzpicture}
\]
It reduces to (3a), hence a contradiction.\vspace{2mm}

\noindent\underline{Case 4. $g \ge 1, |M\setminus P|\ge 3$.}
We construct $T\in\bT_{\cS}$ with CIV for $|M|=3$ and $P=\emptyset$, thus for all cases by Proposition \ref{prop:adding}. In which case, $\cS$ has one of the following:
\begin{itemize}
 \item[(4a)] one boundary component with three marked points;
 \item[(4b)] two boundary components, one has one marked point and the other has two marked points;
 \item[(4c)] three boundary components with one marked point.
\end{itemize}
In the case (4a), we take the following triangle $\triangle$:
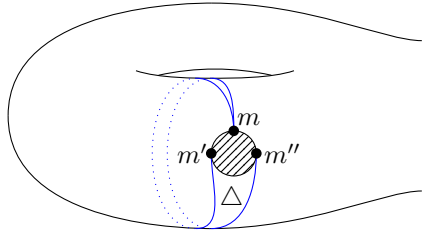
\[
\begin{tikzpicture}[baseline=2mm]
 \coordinate(0)at(0,0); \path(0) ++(90:3mm) coordinate(u); \path(0) ++(180:3mm) coordinate(l); \path(0) ++(0:3mm) coordinate(r);
 \draw[pattern=north east lines] (0) circle (3mm);
 \draw(-3,0.5)..controls(-3,-1)and(-0.5,-1)..(-0.3,-1); \draw(2.5,-0.5)..controls(2,-0.5)and(1,-1)..(-0.3,-1);
 \draw(-3,0.5)..controls(-3,2)and(-0.5,2)..(-0.3,2); \draw(2.5,1.5)..controls(2,1.5)and(1,2)..(-0.3,2);
 \draw(-1.3,1.1)..controls(-1,1)and(-0.5,1)..(-0.3,1); \draw(0.8,1.1)..controls(0.5,1)and(0,1)..(-0.3,1);
 \draw(-1,1.04)..controls(-0.5,1.15)and(0,1.15)..(0.5,1.04);
 \draw[blue](u)to[out=90,in=0](-0.3,1); \draw[blue,dotted](-0.3,1)to[out=180,in=180](-0.3,-1); \draw[blue](r)to[out=-90,in=0](-0.3,-1);
 \draw[blue](u)to[out=90,in=0](-0.5,1); \draw[blue,dotted](-0.5,1)to[out=180,in=180](-0.5,-1); \draw[blue](l)to[out=-90,in=0](-0.5,-1);
 \fill(u)circle(0.07); \fill(l)circle(0.07); \fill(r)circle(0.07);
 \node at(0.2,0.45){$m$}; \node at(-0.55,0){$m'$}; \node at(0.65,0){$m''$}; \node at(-0.03,-0.6){$\triangle$};
\end{tikzpicture}
\]
We obtain a marked surface $\cS'$ from $\cS$ by cutting along the sides of $\triangle$ and removing $\triangle$. Thus its genus is $g-1$ and it has two boundary components with two marked points, and no punctures. We take the following triangles of $\cS'$:
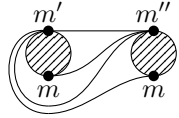
\[
\begin{tikzpicture}[baseline=2mm]
 \coordinate(0)at(-0.7,0); \path(0) ++(90:3mm) coordinate(u); \path(0) ++(-90:3mm) coordinate(d);
 \coordinate(0')at(0.7,0); \path(0') ++(90:3mm) coordinate(u'); \path(0') ++(-90:3mm) coordinate(d');
 \draw[pattern=north east lines] (0) circle (3mm); \draw[pattern=north east lines] (0') circle (3mm);
 \draw(u)--(u');
 \draw(u)..controls(-1.3,0.3)and(-1.3,-0.7)..(-0.7,-0.7); \draw(u')..controls(0,0.3)and(0,-0.7)..(-0.7,-0.7);
 \draw(u)..controls(-1.4,0.4)and(-1.4,-0.8)..(-0.7,-0.8); \draw(d')..controls(0.2,-0.3)and(-0.2,-0.8)..(-0.7,-0.8);
 \draw(u')..controls(0,0.3)and(0,-0.3)..(d);
 \fill(u)circle(0.07); \fill(u')circle(0.07); \fill(d)circle(0.07); \fill(d')circle(0.07);
 \node[above]at(u){$m'$}; \node[above]at(u'){$m''$}; \node[below]at(d){$m$}; \node[below]at(d'){$m$};
\end{tikzpicture}
\]
If $g=1$, then they form a tagged triangulation $T'$ of $\cS'$. Considering $T'$ as triangles on $\cS$, we get $T=T'\cup\triangle\in\bT_{\cS}$ with CIV since their tagged arcs are not loops. If $g>1$, then we can repeat the above process since it reduces to the case (4a) for genus $g-1$. By induction on $g$, we get $T\in\bT_{\cS}$ with CIV. The cases (4b) and (4c) reduce to (4a) by considering the following triangles:
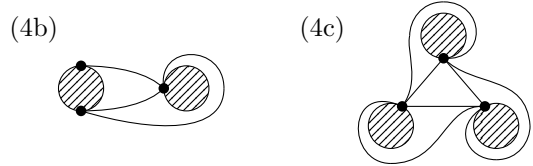
\[
\text{(4b)}
\begin{tikzpicture}[baseline=7mm]
 \coordinate(0)at(-0.7,0); \path(0) ++(90:3mm) coordinate(u); \path(0) ++(-90:3mm) coordinate(d);
 \coordinate(0')at(0.7,0); \path(0') ++(180:3mm) coordinate(l');
 \draw[pattern=north east lines] (0) circle (3mm); \draw[pattern=north east lines] (0') circle (3mm);
 \draw(u)..controls(-0.2,0.3)and(0.2,0.2)..(l'); \draw(d)..controls(-0.2,-0.3)and(0.2,-0.2)..(l');
 \draw(l')..controls(0.3,0.6)and(1.2,0.6)..(1.2,0); \draw(d)..controls(0.3,-0.6)and(1.2,-0.6)..(1.2,0);
 \fill(u)circle(0.07); \fill(d)circle(0.07); \fill(l')circle(0.07);
\end{tikzpicture}
\hspace{10mm}
\text{(4c)}
\begin{tikzpicture}[baseline=12mm]
 \coordinate(0)at(-0.7,0); \path(0) ++(60:3mm) coordinate(l);
 \coordinate(0')at(0.7,0); \path(0') ++(120:3mm) coordinate(r);
 \coordinate(0'')at(0,1.2); \path(0'') ++(-90:3mm) coordinate(u);
 \draw[pattern=north east lines] (0) circle (3mm); \draw[pattern=north east lines] (0') circle (3mm); \draw[pattern=north east lines] (0'') circle (3mm);
 \draw(l)--(r)--(u)--(l);
 \draw(l)..controls(-1.3,0.6)and(-1.3,-0.5)..(-0.7,-0.5); \draw(r)..controls(0,0.3)and(0,-0.5)..(-0.7,-0.5);
 \draw(r)..controls(0.1,0)and(0.3,-0.45)..(0.7,-0.45); \draw(0.7,-0.45)..controls(1,-0.45)and(1.15,-0.2)..(1.15,0); \draw(u)..controls(0.4,0.5)and(1.1,0.7)..(1.15,0);
 \draw(u)..controls(0.6,0.85)and(0.56,1.65)..(0,1.65); \draw(0,1.65)..controls(-0.3,1.65)and(-0.45,1.4)..(-0.45,1.2); \draw(l)..controls(-0.3,0.6)and(-0.5,1)..(-0.45,1.2);
 \fill(l)circle(0.07); \fill(r)circle(0.07); \fill(u)circle(0.07);
\end{tikzpicture}
\]

\noindent\underline{Case 5. $g \ge 1, |M\setminus P|=2, |P| \ge 1$.}
We construct $T\in\bT_{\cS}$ with CIV for $|P|=1$, thus for all cases by Proposition \ref{prop:adding}. In which case, $\cS$ has one of the following:
\begin{itemize}
 \item[(5a)] one boundary component with two marked points;
 \item[(5b)] two boundary components with one marked point.
\end{itemize}
The both cases reduce to (4a) by considering the following triangles:
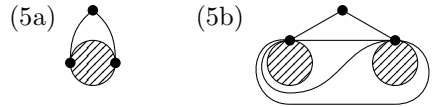
\[
\text{(5a)}\ 
\begin{tikzpicture}[baseline=5mm]
 \coordinate(0)at(0,0); \path(0) ++(90:7mm) coordinate(u); \path(0) ++(180:3mm) coordinate(l); \path(0) ++(0:3mm) coordinate(r);
 \draw[pattern=north east lines] (0) circle (3mm);
 \draw(l)..controls(-0.3,0.3)and(-0.25,0.5)..(u); \draw(r)..controls(0.3,0.3)and(0.25,0.5)..(u);
 \fill(u)circle(0.07); \fill(l)circle(0.07); \fill(r)circle(0.07);
\end{tikzpicture}
\hspace{10mm}
\text{(5b)}
\begin{tikzpicture}[baseline=5mm]
 \coordinate(0)at(-0.7,0); \path(0) ++(90:3mm) coordinate(u);
 \coordinate(0')at(0.7,0); \path(0') ++(90:3mm) coordinate(u');
 \coordinate(p)at(0,0.7);
 \draw[pattern=north east lines] (0) circle (3mm); \draw[pattern=north east lines] (0') circle (3mm);
 \draw(u)--(u')--(p)--(u);
 \draw(u)..controls(-1.2,0.3)and(-1.2,-0.4)..(-0.7,-0.4); \draw(u')..controls(0,0.3)and(-0,-0.4)..(-0.7,-0.4);
 \draw(u)..controls(-1.3,0.4)and(-1.3,-0.55)..(-0.7,-0.55); \draw(u')..controls(1.3,0.3)and(1.3,-0.55)..(0.7,-0.55);
 \draw(-0.7,-0.55)--(0.7,-0.55);
 \fill(u)circle(0.07); \fill(u')circle(0.07); \fill(p)circle(0.07);
\end{tikzpicture}
\]

The cases $1$ to $5$ above cover all cases of $\cS$ with $\partial S\neq\emptyset$. They imply that there is at least one $T \in\bT_{\cS}$ with CIV if and only if $\cS$ is one of the cases 1, 4, and 5. Thus the assertion holds.
\end{proof}

\begin{proof}[Proof of Theorem \ref{thm:CIV cS} for wCIV]
The assertions follow from (1), (2), and Proposition \ref{prop:once-punctured CIV}.
\end{proof}

\begin{example}\label{ex:acyclic}
For an annulus with exactly three marked points on the boundary, we consider the tagged triangulations
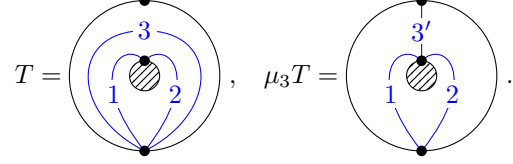
\[
T=
\begin{tikzpicture}[baseline=-1mm]
 \coordinate(0)at(0,0); \coordinate(u)at(0,1); \coordinate(c)at(0,0.2); \coordinate(d)at(0,-1);
 \draw[pattern=north east lines](0)circle(2mm); \draw(0)circle(10mm);
 \draw[blue](d)..controls(-0.8,0)and(-0.3,0.5)..node[fill=white,inner sep=2,pos=0.3]{$1$}(c);
 \draw[blue](d)..controls(0.8,0)and(0.3,0.5)..node[fill=white,inner sep=2,pos=0.3]{$2$}(c);
 \draw[blue](d)..controls(-1,-0.5)and(-1,0.6)..(0,0.6);
 \draw[blue](d)..controls(1,-0.5)and(1,0.6)..(0,0.6)node[fill=white,inner sep=2]{$3$};
 \fill(u)circle(0.07); \fill(c)circle (0.07); \fill(d)circle(0.07);
\end{tikzpicture}\ ,\hspace{3mm}
\mu_3T=
\begin{tikzpicture}[baseline=-1mm]
 \coordinate(0)at(0,0); \coordinate(u)at(0,1); \coordinate(c)at(0,0.2); \coordinate(d)at(0,-1);
 \draw[pattern=north east lines](0)circle(2mm); \draw(0)circle(10mm);
 \draw[blue](d)..controls(-0.8,0)and(-0.3,0.5)..node[fill=white,inner sep=2,pos=0.3]{$1$}(c);
 \draw[blue](d)..controls(0.8,0)and(0.3,0.5)..node[fill=white,inner sep=2,pos=0.3]{$2$}(c);
 \draw[blue](c)--node[fill=white,inner sep=1]{$3'$}(u);
 \fill(u)circle(0.07); \fill(c)circle (0.07); \fill(d)circle(0.07);
\end{tikzpicture}\ .
\]
By Theorem \ref{thm:CIV}, $T$ does not have CIV and $\mu_3T$ has CIV. In fact, for $\gamma\in \mu_3T$ and any tagged arc $\delta$,
\[
(\gamma|\delta)=A_{\gamma,\delta}=\Int(\gamma,\delta).
\]
On the other hand, for another loop $\rho(3)$, we have
\[
(3|\rho(3))=A_{3,\rho(3)}+B_{3,\rho(3)}=2+(-1)=1,\text{ and }\Int(3,\rho(3))=A_{3,\rho(3)}=2.
\]
Thus we have $(T|\rho(3))=(1,1,1)\neq(1,1,2)=\Int(T,\rho(3))$.
\end{example}

\section{Categorification of cluster algebras}\label{sec:categorification}

We denote by $Q_0$ and $Q_1$ the set of vertices and arrows of a quiver $Q$, respectively. We say that an oriented cycle of length one (resp., two) in $Q$ is a \emph{loop} (resp., a \emph{$2$-cycle}).

\subsection{Cluster algebras}

We recall (skew-symmetric) cluster algebras with principal coefficients \cite{FZ07}. For that, we need to prepare some notations. Let $n \in \bZ_{\ge 0}$ and $\cF:=\bQ(t_1,\ldots,t_{2n})$ be the field of rational functions in $2n$ variables over $\bQ$. A \emph{seed with coefficients} is a pair $(\cl,Q)$ consisting of the following data:
\begin{enumerate}
 \item $\cl=(x_1,\ldots,x_n,y_1,\ldots,y_n)$ is a free generating set of $\cF$ over $\bQ$.
 \item $Q$ is a quiver without loops nor $2$-cycles such that $Q_0=\{1,\ldots,2n\}$.
\end{enumerate}
 Then we refer to the tuple $(x_1,\ldots,x_n)$ as the \emph{cluster}, to each $x_i$ as a \emph{cluster variable} and $y_i$ as a \emph{coefficient}. For a seed $(\cl,Q)$ with coefficients, the \emph{mutation $\mu_k(\cl,Q)=(\cl',\mu_kQ)$ at $k$} $(1 \le k \le n)$ is defined as follows:
\begin{enumerate}
 \item $\cl'=(x'_1,\ldots,x'_n,y_1,\ldots,y_n)$ is defined by
 \[
  x_k x'_k = \prod_{(j \rightarrow k)\in Q_1}x_j\prod_{(j \rightarrow k)\in Q_1}y_{j-n}+\prod_{(j \leftarrow k)\in Q_1}x_j\prod_{(j \leftarrow k)\in Q_1}y_{j-n} \ \ \text{and} \ \ x'_i = x_i \ \ \text{if} \ \ i \neq k,
 \]
where $x_{n+1}=\cdots=x_{2n}=1=y_{1-n}=\cdots=y_0$.
 \item $\mu_kQ$ is the quiver obtained from $Q$ by the following steps:\par
 \begin{enumerate}
  \item For any path $i \rightarrow k \rightarrow j$, add an arrow $i \rightarrow j$.
  \item Reverse all arrows incident to $k$.
  \item Remove a maximal set of disjoint $2$-cycles.
 \end{enumerate}
\end{enumerate}

We remark that $\mu_k$ is an involution, that is, we have $\mu_k\mu_k(\cl,Q)=(\cl,Q)$. Moreover, it is elementary that $\mu_k(\cl,Q)$ is also a seed with coefficients.

For a quiver $Q$ without loops nor $2$-cycles such that $Q_0=\{1,\ldots,n\}$, we obtain the quiver $\hat{Q}$ (resp., $\check{Q}$) from $Q$ by adding vertices $\{1',\ldots,n'\}$ and arrows $\{i \rightarrow i' \mid 1 \le i \le n\}$ (resp., $\{i \leftarrow i' \mid 1 \le i \le n\}$). We fix a seed $(\cl=(x_1,\ldots,x_n,y_1,\ldots,y_n),\hat{Q})$ with coefficients, called the \emph{initial seed}. We also call the tuple $(x_1,\ldots,x_n)$ the \emph{initial cluster}, and each $x_i$ the \emph{initial cluster variable}.

\begin{definition}
The \emph{cluster algebra $\cA(Q)=\cA(\cl,\hat{Q})$ with principal coefficients} for the initial seed $(\cl,\hat{Q})$ is a $\bZ$-subalgebra of $\cF$ generated by all cluster variables and coefficients obtained from $(\cl,\hat{Q})$ by sequences of mutations.
\end{definition}

We denote by $\clv Q$ (resp., $\clus Q$) the set of cluster variables (resp., clusters) of $\cA(Q)$. One of the remarkable properties of cluster algebras with principal coefficients is the strongly Laurent phenomenon as follows.

\begin{proposition}[{\cite[Proposition 3.6]{FZ07}}]\label{prop:Laurent phenomenon}
Every $x\in\clv Q$ is expressed by a Laurent polynomial of $x_1,\ldots,x_n$, $y_1,\ldots,y_n$
\[
x=\frac{F(x_1, \ldots,x_n,y_1,\ldots,y_n)}{x_1^{d_1(x)} \cdots x_n^{d_n(x)}},
\]
where $d_i(x) \in \bZ$ and $F(x_1, \ldots, x_n,y_1,\ldots,y_n) \in \bZ[x_1,\ldots, x_n,y_1,\ldots,y_n]$ is not divisible by any $x_i$.
\end{proposition}

\begin{definition}
Using the notations in Proposition \ref{prop:Laurent phenomenon}, we call $d(x):=(d_i(x))_{1 \le i \le n}$ the \emph{denominator vector} of $x$, and $f(x):=(f_i(x))_{1 \le i \le n}$ the \emph{$f$-vector} of $x$, where $f_i(x)$ is the maximal degree of $y_i$ in the \emph{$F$-polynomial} $F(1,\ldots,1,y_1,\ldots,y_n) \in \bZ[y_1,\ldots,y_n]$ of $x$.
\end{definition}

We consider the following property of $\cA(Q)$ or $Q$:
\begin{equation*}\label{eq:d=f}
\tag{d=f} \text{$d(x)=f(x)$ for any $x\in\clv Q \setminus \{\text{initial cluster variables}\}$.}
\end{equation*}
In \cite[Conjecture 7.17]{FZ07}, it was conjectured that any cluster algebra satisfies \eqref{eq:d=f}, but it does not hold in general. In fact, the following counterexample was given in \cite[Subsection 6.4]{FK10}.

\begin{example}\label{ex:counterexample}
Let $Q$ be a quiver
\[
\begin{tikzpicture}[baseline=3mm]
 \node(1)at(-1,0){$1$}; \node(2)at(1,0){$2$}; \node(3)at(0,1){$3$};
 \draw[->](2)--(3); \draw[->](3)--(1);
 \draw[transform canvas={yshift=2pt},->](1)--(2);
 \draw[transform canvas={yshift=-2pt},->](1)--(2);
\end{tikzpicture}
\]
and we compute the following sequence of mutations:
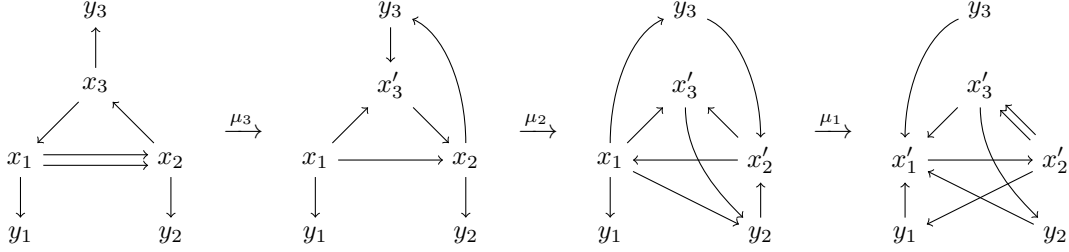
\[
\begin{tikzpicture}[baseline=3mm]
 \node(1)at(-1,0){$x_1$}; \node(2)at(1,0){$x_2$}; \node(3)at(0,1){$x_3$};
 \node(1')at(-1,-1){$y_1$}; \node(2')at(1,-1){$y_2$}; \node(3')at(0,2){$y_3$};
 \draw[transform canvas={yshift=2pt},->](1)--(2); \draw[transform canvas={yshift=-2pt},->](1)--(2);
 \draw[->](2)--(3); \draw[->](3)--(1); \draw[->](1)--(1'); \draw[->](2)--(2'); \draw[->](3)--(3');
\end{tikzpicture}
   \hspace{3mm}\xrightarrow{\mu_3}\hspace{3mm}
\begin{tikzpicture}[baseline=3mm]
 \node(1)at(-1,0){$x_1$}; \node(2)at(1,0){$x_2$}; \node(3)at(0,1){$x_3'$};
 \node(1')at(-1,-1){$y_1$}; \node(2')at(1,-1){$y_2$}; \node(3')at(0,2){$y_3$};
 \draw[->](1)--(2); \draw[<-](2)--(3); \draw[<-](3)--(1); \draw[->](1)--(1'); \draw[->](2)--(2'); \draw[<-](3)--(3');
 \draw[->](2)..controls(1,1)and(0.7,1.7)..(3');
\end{tikzpicture}
   \hspace{3mm}\xrightarrow{\mu_2}\hspace{3mm}
\begin{tikzpicture}[baseline=3mm]
 \node(1)at(-1,0){$x_1$}; \node(2)at(1,0){$x_2'$}; \node(3)at(0,1){$x_3'$};
 \node(1')at(-1,-1){$y_1$}; \node(2')at(1,-1){$y_2$}; \node(3')at(0,2){$y_3$};
 \draw[<-](1)--(2); \draw[->](2)--(3); \draw[<-](3)--(1); \draw[->](1)--(1'); \draw[<-](2)--(2');
 \draw[<-](2)..controls(1,1)and(0.7,1.7)..(3'); \draw[->](1)..controls(-1,1)and(-0.7,1.7)..(3');
 \draw[->](1)--(2'); \draw[->](3)..controls(0,0)and(0.5,-0.5)..(2');
\end{tikzpicture}
   \hspace{3mm}\xrightarrow{\mu_1}\hspace{3mm}
\begin{tikzpicture}[baseline=3mm]
 \node(1)at(-1,0){$x_1'$}; \node(2)at(1,0){$x_2'$}; \node(3)at(0,1){$x_3'$};
 \node(1')at(-1,-1){$y_1$}; \node(2')at(1,-1){$y_2$}; \node(3')at(0,2){$y_3$};
 \draw[transform canvas={yshift=1pt,xshift=1pt},->](2)--(3);
 \draw[transform canvas={yshift=-1pt,xshift=-1pt},->](2)--(3);
 \draw[->](1)--(2); \draw[->](3)--(1); \draw[<-](1)--(1');
 \draw[<-](1)..controls(-1,1)and(-0.7,1.7)..(3');
 \draw[<-](1)--(2'); \draw[->](3)..controls(0,0)and(0.5,-0.5)..(2'); \draw[->](2)--(1');
\end{tikzpicture}
\]
where we put each cluster variable or coefficient on the corresponding vertex and the non-initial cluster variables are expressed by Laurent polynomials
\[
x_3'=\frac{y_3x_1+x_2}{x_3}, x_2'=\frac{y_3x_1^2+x_1x_2+y_2y_3x_3}{x_2x_3}, x_1'=\frac{y_3x_1^2+(y_1y_2y_3^2+1)x_1x_2+y_1y_2y_3x_2^2+y_2y_3x_3}{x_1x_2x_3}.
\]
Then we have that $d(x_1')=(1,1,1) \neq (1,1,2)=f(x_1')$.
\end{example}

\begin{remark}
In this paper, we only consider skew-symmetric cluster algebras. However, we can also consider the property \eqref{eq:d=f} for skew-symmetrizable cluster algebras. In fact, it is known that \eqref{eq:d=f} holds for skew-symmetrizable cluster algebras of rank $2$ \cite[Corollary 10.10]{FZ07} or finite type \cite[Theorem 1.8]{G21}.
\end{remark}

Finally, we prepare the following notation.

\begin{definition}[{\cite{K17}}]
A \emph{reddening sequence} of $Q$ is a sequence $\mu$ of mutations such that $\mu(\hat{Q})=\check{Q}$.
\end{definition}

Note that a reddening sequence is also called a green-to-red sequence \cite{Mu16}, and $Q$ admits a reddening sequence if and only if a seed $(\cl,Q)$ is injective-reachable \cite{Q17}.

\subsection{Categorification of cluster algebras}\label{subsec:categorification}
We briefly recall a categorification of cluster algebras. Let $Q$ be a quiver without loops nor $2$-cycles and $W$ a non-degenerate potential of $Q$ (see \cite{DWZ08}). In particular, such $W$ exists if $K$ is uncountable \cite[Corollary 7.4]{DWZ08}. The mutation of quivers with potentials is defined to be compatible with mutations of quivers, that is, $\mu_k(Q,W)=(\mu_kQ,W')$ for $k\in Q_0$ and some non-degenerate potential $W'$ of $\mu_kQ$. Moreover, $(Q,W)$ defines the Jacobian algebra $J=J(Q,W)$, the Ginzburg differential graded algebra $\Gamma=\Gamma_{Q,W}$, and the generalized cluster category $\cC=\cC_{Q,W}$ (see \cite{A09,DWZ08,G06,K08,K11}). The category $\cC$ is a triangulated category with a rigid object $\Gamma$, that is $\Hom_{\cC}(\Gamma,\Sigma\Gamma)=0$, and $\End_{\cC}(\Gamma)^{\rm op}\simeq J$, where $\Sigma$ is the suspension functor.

For the idempotent $e_i$ associated with $i\in Q_0$, let $\Gamma_i=\Gamma e_i$ and $\Gamma'_i=\Gamma_{\mu_k(Q,W)}e_i$. We take the cone $\Gamma_k^{\ast}$ of the morphism in $\cC$
\[
\Gamma_k\rightarrow\bigoplus_{(k\leftarrow i)\in Q_1}\Gamma_i
\]
whose components are given by the right multiplications of arrows. We say that $\mu_k\Gamma:=(\Gamma/\Gamma_k)\oplus\Gamma_k^{\ast}$ is the \emph{mutation of $\Gamma$ at $k\in Q_0$}. Moreover, there is a triangle equivalence
\[
\cC_{\mu_k(Q,W)}\rightarrow\cC
\]
which sends $\Gamma'_i$ to $\Gamma_i$ for $i\neq k$ and $\Gamma'_k$ to $\Gamma_k^{\ast}$ \cite[Theorem 3.2]{KY11} (see also \cite[Subsection 2.6]{Pl11c}). Then we consider sequences of mutations from $\Gamma$ as follows: For $k_1,\ldots,k_j \in Q_0$, there is a triangle equivalence from $\cC_{\mu_{k_j}\cdots\mu_{k_1}(Q,W)}$ to $\cC$ given by the sequence of the above triangle equivalences
\[
\cC_{\mu_{k_j}\cdots\mu_{k_1}(Q,W)}\rightarrow\cdots\rightarrow\cC_{\mu_{k_1}(Q,W)}\rightarrow\cC,
\]
and $\mu_{k_j}\cdots\mu_{k_1}\Gamma$ is defined by the image of $\Gamma_{\mu_{k_j}\cdots\mu_{k_1}(Q,W)}$. Note that $\mu_{k_j}\cdots\mu_{k_1}\Gamma$ is also rigid. We denote by $\rigid^{\Gamma}\cC$ the set of isomorphism classes of rigid objects in $\cC$ obtained from $\Gamma$ by sequences of mutations, and by $\irigid^{\Gamma}\cC$ the set of isomorphism classes of indecomposable direct summands of $\Gamma'\in\rigid^{\Gamma}\cC$. Let $\mod J$ be the category of finite dimensional left $J$-modules.

\begin{proposition}[{\cite[Lemma 3.2 and Subsection 3.3]{Pl11c}}]\label{prop:eq mod}
There is a full subcategory $\cD$ of $\cC$ including all objects in $\irigid^{\Gamma}\cC$ with an equivalence of categories
\[
\Hom_{\cC}(\Sigma^{-1}\Gamma,-) : \cD / (\Gamma) \xrightarrow{\sim} \mod\End_{\cC}(\Sigma^{-1}\Gamma)^{\rm op}\simeq\mod J,
\]
where $(\Gamma)$ is the ideal of morphisms of $\cC$ which factor through $\Gamma$.
\end{proposition}

Let $P_i=J e_i$ for $i\in Q_0$. By $\Hom_{\cC}(\Gamma_i,\Sigma\Gamma)=0$ and Proposition \ref{prop:eq mod}, we have
{\setlength\arraycolsep{0.5mm}
\begin{eqnarray*}
\Hom_{J}(P_i,\Hom_{\cC}(\Sigma^{-1}\Gamma,X))&\simeq&\Hom_{J}(\Hom_{\cC}(\Sigma^{-1}\Gamma,\Sigma^{-1}\Gamma_i),\Hom_{\cC}(\Sigma^{-1}\Gamma,X))\\
&\simeq&\Hom_{\cC}(\Sigma^{-1}\Gamma_i,X).
\end{eqnarray*}}
In particular, taking $X=\Sigma^{-1}\Gamma_i$, we have
\begin{equation}\label{eq:End}
\End_{J}(P_i)\simeq\End_{\cC}(\Gamma_i).
\end{equation}

The following is a main result of the categorification of cluster algebras.

\begin{theorem}\label{thm:bijection-sC}
\begin{enumerate}
\item\cite[Theorem 4.1]{Pl11c}\cite[Corollary 3.5]{CKLP13}
There is a bijection
\[
\sC_{\Gamma} : \irigid^{\Gamma}\cC \leftrightarrow \clv Q
\]
which sends $\Gamma_i$ to the initial cluster variable $x_i$ for $i\in Q_0$ and induces a bijection
\[
\sC_{\Gamma} : \rigid^{\Gamma}\cC \leftrightarrow \clus Q
\]
which sends $\Gamma$ to the initial cluster and commutes with mutations.
\item\cite[Proposition 3.1]{DWZ10}\cite[Subsection 3.3]{Pl11a} For $X\in\irigid^{\Gamma}$, we have
\[
f(\sC_{\Gamma}(X))=\udim\Hom_{\cC}(\Sigma^{-1}\Gamma,X),
\]
where the right side is the dimension vector of the $J$-module $\Hom_{\cC}(\Sigma^{-1}\Gamma,X)$.
\end{enumerate}
\end{theorem}

Proposition \ref{prop:eq mod} and Theorem \ref{thm:bijection-sC} give the map in the introduction
\begin{equation*}\label{eq:bijection-sM}
\sM=\sM_{Q,W}:=\Hom_{\cC}(\Sigma^{-1}\Gamma,\sC_{\Gamma}^{-1}(-)) : \clv Q \setminus \{\text{initial cluster variables}\} \rightarrow \mod J,
\end{equation*}
and $f(x)=\udim\sM(x)$ for $x\in\clv Q \setminus \{\text{initial cluster variables}\}$. This implies the following.

\begin{corollary}\label{cor:d=dim d=f}
The quiver $Q$ satisfies \eqref{eq:d=dim} if and only if it satisfies \eqref{eq:d=f}.
\end{corollary}

On the other hand, if $J$ is finite dimensional, then $\cC$ is a Hom-finite Krull-Schmidt $2$-Calabi-Yau triangulated category and $\Gamma$ is a cluster tilting object \cite[Theorem 2.1]{A09}. From this viewpoint, it is important to check the finite dimensionality of $J$ and one way to do it is given as a consequence of \cite[Lemma 2.17]{Pl11c}.

\begin{proposition}[{\cite[Proposition 8.1]{BDP14}}]\label{prop:fin dim}
If $Q$ admits a reddening sequence, then $J(Q,W')$ is finite dimensional for any non-degenerate potential $W'$ of $Q$.
\end{proposition}

\section{Cluster algebras from triangulated surfaces}\label{sec:CA from TS}

\subsection{Cluster algebras from triangulated surfaces}\label{subsec:CA TS}

Let $\cS$ be a marked surface. To $T\in\bT_{\cS}$, we associate a quiver $\bar{Q}_T$ with $(\bar{Q}_T)_0=T$ and whose arrows correspond with angles between tagged arcs of $T$ as in Table \ref{table:puzzle pieces to quivers}. We obtain a quiver $Q_T$ without loops nor $2$-cycles from $\bar{Q}_T$ by removing $2$-cycles. This construction commutes with flips and mutations, that is, $\mu_{\gamma}Q_T=Q_{\mu_{\gamma}T}$ for $\gamma\in T$ \cite[Proposition 4.8 and Lemma 9.7]{FoST08}. This implies that a quiver $Q$ is obtained from $Q_T$ by a sequence of mutations if and only if $Q=Q_{T'}$ for some $T'\in\bT_{\cS}^T$.
\renewcommand{\arraystretch}{2}
{\begin{table}[ht]
\begin{tabular}{c|c|c|c|c}
$\triangle$
&
\begin{tikzpicture}[baseline=0mm]
 \coordinate(l)at(-150:1); \coordinate(r)at(-30:1); \coordinate(u)at(90:1);
 \draw(u)--node[left]{$1$}(l)--node[below]{$3$}(r)--node[right]{$2$}(u);
 \fill(u)circle(0.07); \fill(l)circle(0.07); \fill(r)circle(0.07);
\end{tikzpicture}
&
\begin{tikzpicture}[baseline=-2mm]
 \coordinate(c)at(0,0); \coordinate(u)at(0,1); \coordinate(d)at(0,-1);
 \draw(d)to[out=180,in=180]node[left]{$1$}(u);
 \draw(d)to[out=0,in=0]node[right]{$2$}(u);
 \draw(d)to[out=150,in=-150]node[fill=white,inner sep=1]{$3$}(c);
 \draw(d)to[out=30,in=-30]node[pos=0.8]{\rotatebox{40}{\footnotesize $\bowtie$}}node[fill=white,inner sep=1]{$4$}(c);
 \fill(c)circle(0.07); \fill(u)circle(0.07); \fill(d)circle(0.07);
\end{tikzpicture}
&
\begin{tikzpicture}[baseline=-2mm]
 \coordinate(l)at(-0.5,0); \coordinate(r)at(0.5,0); \coordinate(d)at(0,-1);
 \draw(0,0)circle(1); \node at(0,0.7){$1$};
 \draw(d)to[out=170,in=-130]node[fill=white,inner sep=1]{$2$}(l);
 \draw(d)to[out=95,in=0]node[left,inner sep=1,pos=0.4]{$3$}node[pos=0.8]{\rotatebox{60}{\footnotesize $\bowtie$}}(l);
 \draw(d)to[out=85,in=180]node[right,inner sep=1,pos=0.4]{$4$}(r);
 \draw(d)to[out=10,in=-50] node[pos=0.8]{\rotatebox{10}{\footnotesize $\bowtie$}}node[fill=white,inner sep=1]{$5$}(r);
 \fill(l)circle(0.07); \fill(r)circle(0.07); \fill(d)circle(0.07);
\end{tikzpicture}
&
\begin{tikzpicture}[baseline=1mm]
 \coordinate(c)at(0,0); \coordinate(u)at(90:1); \coordinate(r)at(-30:1); \coordinate(l)at(210:1);
 \draw(c)to[out=60,in=120,relative]node[fill=white,inner sep=1]{$1$}(u);
 \draw(c)to[out=-60,in=-120,relative]node[pos=0.8]{\rotatebox{40}{\footnotesize $\bowtie$}}node[fill=white,inner sep=1]{$2$}(u);
 \draw(c)to[out=60,in=120,relative]node[fill=white,inner sep=1]{$3$}(l);
 \draw(c)to[out=-60,in=-120,relative] node[pos=0.8]{\rotatebox{160}{\footnotesize $\bowtie$}}node[fill=white,inner sep=1]{$4$}(l);
 \draw(c)to[out=60,in=120,relative]node[fill=white,inner sep=1]{$5$}(r);
 \draw(c)to[out=-60,in=-120,relative]node[pos=0.8]{\rotatebox{-80}{\footnotesize $\bowtie$}}node[fill=white,inner sep=1]{$6$}(r);
 \fill(c)circle(0.07); \fill(u)circle(0.07); \fill(l)circle(0.07); \fill(r)circle(0.07);
\end{tikzpicture}
\\\hline
$Q_{\triangle}$
&
\begin{tikzpicture}[baseline=-3mm,scale=0.8]
 \node(1)at(150:1){$1$}; \node(2)at(30:1){$2$}; \node(3)at(-90:1){$3$};
 \draw[->](3)--(2); \draw[->](2)--(1); \draw[->](1)--(3);
\end{tikzpicture}
&
\begin{tikzpicture}[baseline=-2mm]
 \node(1)at(-0.7,0.5){$1$}; \node(2)at(0.7,0.5){$2$}; \node(3)at(0,-0.2){$3$}; \node(4)at(0,-0.8){$4$};
 \draw[->](3)--(2); \draw[->](2)--(1); \draw[->](1)--(3); \draw[->](4)--(2); \draw[->](1)--(4);
\end{tikzpicture}
&
\begin{tikzpicture}[baseline=-1mm]
 \node(1)at(0,0.8){$1$}; \node(2)at(-170:1){$2$}; \node(3)at(-130:1){$3$};
 \node(4)at(-10:1){$4$}; \node(5)at(-50:1){$5$};
 \draw[->] (1)--(2); \draw[->] (2)--(4); \draw[->] (4)--(1);
 \draw[->] (1)--(3); \draw[->] (3)--(4);
 \draw[->] (2)--(5); \draw[->] (5)--(1); \draw[->] (3)--(5);
\end{tikzpicture}
&
\begin{tikzpicture}[baseline=0mm]
 \node(1)at(90:0.5){$1$}; \node(2)at(90:1.1){$2$}; \node(3)at(-30:0.5){$5$};
 \node(4)at(-30:1.1){$6$}; \node(5)at(210:0.5){$3$}; \node(6)at(210:1.1){$4$};
 \draw[->](1)--(3); \draw[->](3)--(5); \draw[->](5)--(1);
 \draw[->] (2) to [out=40,in=140,relative] (4); \draw[->] (4) to [out=40,in=140,relative] (6); \draw[->] (6) to [out=40,in=140,relative] (2);
 \draw[->] (2) to [out=40,in=140,relative] (3); \draw[->] (4) to [out=40,in=140,relative] (5); \draw[->] (6) to [out=40,in=140,relative] (1);
 \draw[->] (1) to [out=40,in=140,relative] (4); \draw[->] (3) to [out=40,in=140,relative] (6); \draw[->] (5) to [out=40,in=140,relative] (2);
\end{tikzpicture}
\end{tabular}\vspace{3mm}
\caption{The quiver $Q_{\triangle}$ associated with each puzzle piece $\triangle$}\label{table:puzzle pieces to quivers}
\end{table}}

\begin{example}\label{ex:quivers from sS}
For marked surfaces in Theorem \ref{thm:d=dim cS}(2) (or Theorem \ref{thm:CIV cS}(2)), it is easy to see the associated quivers as in Table \ref{table:examples of quivers}. In particular, $\cS$ is a polygon without puncture (resp., with exactly one puncture) if and only if there is a $T\in\bT_{\cS}$ such that $Q_T$ is a Dynkin quiver of type $A$ (resp., $D$). Moreover, $\cS$ is an annulus with exactly two marked points (resp., a torus with exactly one puncture) if and only if all associated quivers are Kronecker (resp., Markov) quivers. These immediately give Corollary \ref{cor:analogue2} from Theorem \ref{thm:d=dim cS}(2).
\renewcommand{\arraystretch}{2}
{\begin{table}[ht]
\begin{tabular}{c|c|c|c|c}
$T$
&
\begin{tikzpicture}[baseline=0mm]
 \coordinate(c)at(0,0); \coordinate(u)at(0,1); \coordinate(d)at(0,-1);
 \coordinate(ru)at(30:1); \coordinate(rd)at(-30:1); \coordinate(lu)at(150:1); \coordinate(ld)at(-150:1);
 \draw(c)circle(10mm);
 \draw[blue](u)--node[fill=white,inner sep=1]{$1$}(ld) (u)--node[fill=white,inner sep=1]{$2$}(d) (u)--node[fill=white,inner sep=1]{$3$}(rd);
 \fill(u)circle(0.07); \fill(d)circle(0.07); \fill(ru)circle(0.07); \fill(rd)circle(0.07); \fill(lu)circle(0.07); \fill(ld)circle(0.07);
 \node at(0,-1.1){};
\end{tikzpicture}
&
\begin{tikzpicture}[baseline=0mm]
 \coordinate(c)at(0,0); \coordinate(p)at(0,-0.3);
 \coordinate(ru)at(45:1); \coordinate(rd)at(-45:1); \coordinate(lu)at(135:1); \coordinate(ld)at(-135:1);
 \draw(c)circle(10mm);
 \draw[blue](ld)..controls(-1,0.2)and(-0.2,1)..node[fill=white,inner sep=1]{$1$}(ru);
 \draw[blue](ld)..controls(-0.6,0.4)and(0.6,0.4)..node[fill=white,inner sep=1]{$2$}(rd);
 \draw[blue](ld)--node[fill=white,inner sep=1]{$3$}(p)--node[fill=white,inner sep=1]{$4$}(rd);
 \fill(p)circle(0.07); \fill(ru)circle(0.07); \fill(rd)circle(0.07); \fill(lu)circle(0.07); \fill(ld)circle(0.07);
\end{tikzpicture}
&
\begin{tikzpicture}[baseline=0mm]
 \coordinate(c)at(0,0); \coordinate(cu)at(0,0.2); \coordinate(d)at(0,-1);
 \draw[pattern=north east lines](c)circle(2mm); \draw(c)circle(10mm);
 \draw[blue](d)..controls(-1,0)and(-0.3,0.6)..node[fill=white,inner sep=2,pos=0.3]{$1$}(cu);
 \draw[blue](d)..controls(1,0)and(0.3,0.6)..node[fill=white,inner sep=2,pos=0.3]{$2$}(cu);
 \fill(d)circle(0.07); \fill(cu)circle(0.07);
\end{tikzpicture}
&
\begin{tikzpicture}[baseline=0mm]
 \coordinate(p)at(0,-0.4);
 \draw(-1.3,0)to[out=90,in=90](1.3,0); \draw(-1.3,0)to[out=-90,in=-90](1.3,0);
 \draw(-0.5,0.1)to[out=-30,in=-150](0.5,0.1); \draw(-0.3,0)to[out=30,in=150](0.3,0);
 \draw[blue](p)..controls(1.3,-0.3)and(1.3,0.6)..(0,0.6);
 \draw[blue](p)..controls(-1.3,-0.3)and(-1.3,0.6)..node[fill=white,inner sep=1,pos=0.3]{$1$}(0,0.6);
 \draw[blue](p)to[out=90,in=180](0.15,-0.04); \draw[blue](p)to[out=-90,in=180](0.15,-0.75);
 \draw[blue,dotted](0.3,-0.4)to[out=90,in=0](0.15,-0.04); \draw[blue,dotted](0.3,-0.4)to[out=-90,in=0](0.15,-0.75);
 \draw[blue](p)..controls(1.3,0.2)and(0.3,0.6)..node[fill=white,inner sep=1,pos=0.6]{$3$}(-0.3,0);
 \draw[blue](p)..controls(-0.3,-0.6)and(-0.4,-0.7)..(-0.5,-0.71);
 \draw[blue,dotted](-0.3,0)..controls(-0.5,-0.3)and(-0.6,-0.7)..(-0.5,-0.71); \node[blue]at(0.2,-0.95){$2$};
 \fill(p)circle(0.07);
\end{tikzpicture}
\\\hline
$Q_T$
&
\begin{tikzpicture}[baseline=-1mm,scale=0.8]
 \node(1)at(-1,0){$1$}; \node(2)at(0,0){$2$}; \node(3)at(1,0){$3$};
 \draw[->](3)--(2); \draw[->](2)--(1);
\end{tikzpicture}
&
\begin{tikzpicture}[baseline=0mm]
 \node(1)at(0,0.9){$1$}; \node(2)at(0,0){$2$}; \node(3)at(-0.7,-0.7){$3$}; \node(4)at(0.7,-0.7){$4$};
 \draw[->](1)--(2); \draw[->](2)--(3); \draw[->](4)--(2);
 \node at(0,1.1){};
\end{tikzpicture}
&
\begin{tikzpicture}[baseline=-1mm]
 \node(1)at(-0.5,0){$1$}; \node(2)at(0.5,0){$2$};
 \draw[transform canvas={yshift=2pt},->](1)--(2); \draw[transform canvas={yshift=-2pt},->](1)--(2);
\end{tikzpicture}
&
\begin{tikzpicture}[baseline=1mm]
 \node(1)at(210:1){$1$}; \node(2)at(-30:1){$2$}; \node(3)at(90:1){$3$};
 \draw[transform canvas={yshift=2pt},->](1)--(2); \draw[transform canvas={yshift=-2pt},->](1)--(2);
 \draw[transform canvas={xshift=-1.8pt,yshift=-1.3pt},->](2)--(3); \draw[transform canvas={xshift=1.8pt,yshift=1.3pt},->](2)--(3);
 \draw[transform canvas={xshift=-1.8pt,yshift=1.3pt},->](3)--(1); \draw[transform canvas={xshift=1.8pt,yshift=-1.3pt},->](3)--(1);
\end{tikzpicture}
\end{tabular}\vspace{3mm}
\caption{The quiver $Q_T$ associated with a tagged triangulation $T$ of each marked surface in Theorem \ref{thm:d=dim cS}(2)}
\label{table:examples of quivers}
\end{table}}
\end{example}

The associated cluster algebra $\cA(Q_T)$ has the following properties.

\begin{theorem}\label{thm:bijection-x}
Let $T\in\bT_{\cS}$.
\begin{enumerate}
\item\cite[Theorem 7.11]{FoST08}\cite[Theorem 6.1]{FT18} There is a bijection
\[
x_T : \bA_{\cS}^T \leftrightarrow \clv Q_T.
\]
Moreover, it induces a bijection
\[
x_T : \bT_{\cS}^T \leftrightarrow \clus Q_T
\]
which sends $T$ to the initial cluster of $\cA(Q_T)$ and commutes with flips and mutations.
\item\cite[(7.7)]{FZ07}\cite[Lemma 9.20]{FoST08} For $\gamma \in \bA_{\cS}^T$, we have
\[
d(x_T(\gamma))=(T | \gamma).
\]
\item\cite[Theorem 7]{Y19}\label{thm:f=Int} For $\gamma \in \bA_{\cS}^T$, we have
\[
f(x_T(\gamma))=\Int(T,\gamma).
\]
\item\cite[Theorem 1.1]{Mi17}\label{thm:mgs} If $\cS$ is not a closed surface with exactly one puncture, then $Q_T$ admits a reddening sequence.
\end{enumerate}
\end{theorem}

By Theorem \ref{thm:bijection-x}, we can apply the results in Section \ref{sec:cS} to cluster algebras. In particular, we get Theorems \ref{thm:d=dim} and \ref{thm:d=dim cS} from the following corollary.

\begin{corollary}\label{cor:wCIV d=dim}
For $T\in\bT_{\cS}$, it has wCIV if and only if $Q_T$ satisfies \eqref{eq:d=dim}.
\end{corollary}

\begin{proof}
The assertion follows from Corollary \ref{cor:d=dim d=f} and Theorem \ref{thm:bijection-x}.
\end{proof}

\begin{proof}[Proof of Theorem \ref{thm:d=dim}]
The assertion follows from Theorems \ref{thm:CIV}, \ref{thm:wCIV}, and Corollary \ref{cor:wCIV d=dim}.
\end{proof}

\begin{proof}[Proof of Theorem \ref{thm:d=dim cS}]
The assertions follows from Theorem \ref{thm:CIV cS} and Corollary \ref{cor:wCIV d=dim}.
\end{proof}

\begin{example}
For the tagged triangulation $T$ in Example \ref{ex:acyclic}, $Q_T$ is the quiver in Example \ref{ex:counterexample}. Moreover, the sequence of mutations in Example \ref{ex:counterexample} corresponds with the sequence of flips
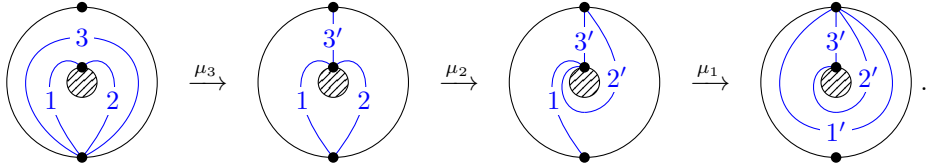
\[
\begin{tikzpicture}[baseline=-1mm]
 \coordinate(0)at(0,0); \coordinate(u)at(0,1); \coordinate(c)at(0,0.2); \coordinate(d)at(0,-1);
 \draw[pattern=north east lines](0)circle(2mm); \draw(0)circle(10mm);
 \draw[blue](d)..controls(-0.8,0)and(-0.3,0.5)..node[fill=white,inner sep=2,pos=0.3]{$1$}(c);
 \draw[blue](d)..controls(0.8,0)and(0.3,0.5)..node[fill=white,inner sep=2,pos=0.3]{$2$}(c);
 \draw[blue](d)..controls(-1,-0.5)and(-1,0.6)..(0,0.6);
 \draw[blue](d)..controls(1,-0.5)and(1,0.6)..(0,0.6)node[fill=white,inner sep=2]{$3$};
 \fill(u)circle(0.07); \fill(c)circle (0.07); \fill(d)circle(0.07);
\end{tikzpicture}
 \hspace{3mm}\xrightarrow{\mu_3}\hspace{3mm}
\begin{tikzpicture}[baseline=-1mm]
 \coordinate(0)at(0,0); \coordinate(u)at(0,1); \coordinate(c)at(0,0.2); \coordinate(d)at(0,-1);
 \draw[pattern=north east lines](0)circle(2mm); \draw(0)circle(10mm);
 \draw[blue](d)..controls(-0.8,0)and(-0.3,0.5)..node[fill=white,inner sep=2,pos=0.3]{$1$}(c);
 \draw[blue](d)..controls(0.8,0)and(0.3,0.5)..node[fill=white,inner sep=2,pos=0.3]{$2$}(c);
 \draw[blue](c)--node[fill=white,inner sep=1]{$3'$}(u);
 \fill(u)circle(0.07); \fill(c)circle (0.07); \fill(d)circle(0.07);
\end{tikzpicture}
 \hspace{3mm}\xrightarrow{\mu_2}\hspace{3mm}
\begin{tikzpicture}[baseline=-1mm]
 \coordinate(0)at(0,0); \coordinate(u)at(0,1); \coordinate(c)at(0,0.2); \coordinate(d)at(0,-1);
 \draw[pattern=north east lines](0)circle(2mm); \draw(0)circle(10mm);
 \draw[blue](d)..controls(-0.8,0)and(-0.3,0.5)..node[fill=white,inner sep=2,pos=0.3]{$1$}(c);
 \draw[blue](u)..controls(0.8,0)and(0.3,-0.4)..node[fill=white,inner sep=2,pos=0.4]{$2'$}(0,-0.4);
 \draw[blue](c)..controls(-0.4,0.3)and(-0.4,-0.4)..(0,-0.4);
 \draw[blue](c)--node[fill=white,inner sep=1,pos=0.45]{$3'$}(u);
 \fill(u)circle(0.07); \fill(c)circle (0.07); \fill(d)circle(0.07);
\end{tikzpicture}
 \hspace{3mm}\xrightarrow{\mu_1}\hspace{3mm}
\begin{tikzpicture}[baseline=-1mm]
 \coordinate(0)at(0,0); \coordinate(u)at(0,1); \coordinate(c)at(0,0.2); \coordinate(d)at(0,-1);
 \draw[pattern=north east lines](0)circle(2mm); \draw(0)circle(10mm);
 \draw[blue](u)..controls(-1,0.5)and(-1,-0.65)..(0,-0.65);
 \draw[blue](u)..controls(1,0.5)and(1,-0.65)..(0,-0.65)node[fill=white,inner sep=2]{$1'$};
 \draw[blue](u)..controls(0.8,0)and(0.3,-0.4)..node[fill=white,inner sep=2,pos=0.4]{$2'$}(0,-0.4);
 \draw[blue](c)..controls(-0.4,0.3)and(-0.4,-0.4)..(0,-0.4);
 \draw[blue](c)--node[fill=white,inner sep=1,pos=0.45]{$3'$}(u);
 \fill(u)circle(0.07); \fill(c)circle (0.07); \fill(d)circle(0.07);
\end{tikzpicture}\ .
\]
In particular, we have $\rho(3)=1'$ and $x_T(i')=x_{i}'$ for $i\in\{1,2,3\}$, where $x_{i}'$ is the non-initial cluster variable of $\cA(Q_T)$ appearing in Example \ref{ex:counterexample}. Then we can see $d(x_T(1'))=(T|1')=(1,1,1)\neq(1,1,2)=f(x_T(1'))=\Int(T,1')$.

\end{example}

\subsection{Categorification on triangulated surfaces}\label{subsec:categorification TS}

Let $\cS$ be a marked surface and $T\in\bT_{\cS}$. The following property is given as a consequence of the classification of non-degenerate potentials of $Q_T$ studied in \cite{GG15,GLS16,GLM22,Lab09,Lab16,Lad12} (see also \cite{Lab16survey}). It also follows from Proposition \ref{prop:fin dim} and Theorem \ref{thm:bijection-x}(4).

\begin{theorem}[Finite dimensionality]\label{thm:fin dim}
If $\cS$ is not a closed surface with exactly one puncture, then $J(Q_T,W)$ is finite dimensional for any non-degenerate potential $W$ of $Q_T$.
\end{theorem}

We remark that if $\cS$ is a closed surface with exactly one puncture, then there is a non-degenerate potential $W$ such that $J(Q,W)$ is infinite dimensional \cite[Proof of Proposition 9.13]{GLS16}.

Let $W$ be a non-degenerate potential of $Q_T$. In the rest of this paper, we denote $J=J(Q_T,W)$, $\Gamma=\Gamma_{Q_T,W}$, $\cC=\cC_{Q_T,W}$, $P_{\gamma}=J e_{\gamma}$, and $\Gamma_{\gamma}=\Gamma e_{\gamma}$ for $\gamma \in T=(Q_T)_0$. By Theorems \ref{thm:bijection-sC} and \ref{thm:bijection-x}, we have the bijection
\begin{equation}\label{eq:sX}
\sX_{T}:=\sC_{\Gamma}^{-1}x_{T} : \bA_{\cS} \leftrightarrow \irigid^{\Gamma}\cC
\end{equation}
which sends $\gamma \in T$ to $\Gamma_{\gamma}$ and $\Int(T,\delta)=\udim\Hom_{\cC}(\Sigma^{-1}\Gamma,\sX_{T}(\delta))$. The following theorem was given in \cite{BQ15} under the assumption that $J$ is finite dimensional, but the assumption automatically holds by Theorem \ref{thm:fin dim}.

\begin{theorem}[{\cite[Theorem 3.8]{BQ15}}]\label{thm:rotation=shift}
If $\cS$ is not a closed surface with exactly one puncture, then we have $\sX_{T}(\rho(\gamma))=\Sigma\sX_{T}(\gamma)$ for $\gamma \in \bA_{\cS}$.
\end{theorem}

We are ready to give equivalences including Theorem \ref{thm:analogue1}.

\begin{theorem}\label{thm:TFAE final}
If $\cS$ is not a closed surface with exactly one puncture, then the following are equivalent:
\begin{enumerate}
\item $Q_T$ satisfies \eqref{eq:d=dim};
\item $\End_{\cC}(\Gamma_{\gamma})\simeq K$ for any $\gamma \in T$;
\item $\End_{J}(P_{\gamma})\simeq K$ for any $\gamma \in T$.
\end{enumerate}
\end{theorem}

\begin{proof}
Under the assumption, CIV and wCIV are equivalent. Then (1) is equivalent to $\Int(\gamma,\rho^{-1}(\gamma))=1$ for any $\gamma\in T$ by Theorem \ref{thm:CIV} and Corollary \ref{cor:wCIV d=dim}. By \eqref{eq:sX} and Theorem \ref{thm:rotation=shift}, we have
{\setlength\arraycolsep{0.5mm}
\begin{eqnarray*}
\Int(\gamma,\rho^{-1}(\gamma))
&=&\dim\Hom_{\cC}(\Sigma^{-1}\Gamma_{\gamma},\sX_T(\rho^{-1}(\gamma)))\\
&=&\dim\Hom_{\cC}(\Sigma^{-1}\Gamma_{\gamma},\Sigma^{-1}\sX_T(\gamma))\\
&=&\dim\End_{\cC}(\Gamma_{\gamma}).
\end{eqnarray*}}\hspace{-1mm}
Therefore, (1) and (2) are equivalent. The equivalence between (2) and (3) follows from \eqref{eq:End}.
\end{proof}

We also consider Corollary \ref{cor:analogue2} from the viewpoint of cluster categories. The following theorem was also given in \cite{Y20} under the assumption that $J$ is finite dimensional, but the assumption automatically holds by Theorem \ref{thm:fin dim}.

\begin{theorem}[{\cite[Theorem 1.4]{Y20}}]\label{thm:all rigid}
If $\cS$ is not a closed surface with exactly one puncture, then all indecomposable rigid objects in $\cC$ are contained in $\irigid^{\Gamma}\cC$.
\end{theorem}

We add one condition to Corollary \ref{cor:analogue2} under the assumption that $\cS$ is not closed surface with exactly one puncture.

\begin{proposition}\label{prop:characterization}
If $\cS$ is not closed surface with exactly one puncture, then the following are equivalent:
\begin{enumerate}
\item $Q_T$ satisfies \eqref{eq:d=dim} for any $T\in\bT_{\cS}$;
\item $Q_T$ is either Dynkin type or a Kronecker quiver for any $T\in\bT_{\cS}$;
\item $\End_{\cC}(X)\simeq K$ for any indecomposable rigid object $X\in\cC$.
\end{enumerate}
\end{proposition}

\begin{proof}
By the assumption and Theorem \ref{thm:transitivity of flips}, we have $\bT_{\cS}=\bT_{\cS}^T$. Then the equivalence between (1) and (2) is just Corollary \ref{cor:analogue2}.

For a sequence $\mu$ of mutations, the definition of mutations of $\Gamma$ gives that
\[
\End_{\cC}(\mu\Gamma)\simeq\End_{\cC_{\mu(Q_T,W)}}(\Gamma_{\mu(Q_T,W)}).
\]
Therefore, by Theorem \ref{thm:TFAE final}, $\mu(Q_T)=Q_{\mu T}$ satisfies \eqref{eq:d=dim} if and only if $\End_{\cC_{\mu(Q_T,W)}}(X)\simeq K$ for any indecomposable direct summand $X$ of $\Gamma_{\mu(Q_T,W)}$, which is also equivalent to $\End_{\cC}(Y)\simeq K$ for any indecomposable direct summand $Y$ of $\mu\Gamma$. Thus the equivalence between (1) and (3) follows from Theorem \ref{thm:all rigid}.
\end{proof}

\section{Example}\label{sec:example}

We consider an example which is not covered by \cite{BMR09}. For an annulus with three marked points on the boundary and one puncture, we consider the following tagged triangulation and the associated quiver:
\[
\begin{tikzpicture}[baseline=-6mm,scale=3]
 \node at(-1.4,-0.2){$T=$}; \node at(1.3,-0.2){$,$};
 \coordinate(0)at(0,0); \coordinate(u)at(0,1); \coordinate(c)at(0,0.2); \coordinate(d)at(0,-1);
 \draw[pattern=north east lines](0)circle(2mm); \draw(0,-0.2)circle(12mm);
 \draw[blue](d)..controls(-0.8,0)and(-0.3,0.5)..node[fill=white,inner sep=2,pos=0.3]{$1$}(c);
 \draw[blue](d)..controls(0.8,0)and(0.3,0.5)..node[fill=white,inner sep=2,pos=0.3]{$2$}(c);
 \draw[blue](d)..controls(-1,-1)and(-1,0.6)..(0,0.6);
 \draw[blue](d)..controls(1,-1)and(1,0.6)..(0,0.6)node[fill=white,inner sep=2]{$3$};
 \draw[blue](u)..controls(-1.5,0.5)and(-1.1,-1.5)..node[fill=white,inner sep=2]{$5$}(d);
 \draw[blue](u)..controls(1.5,0.5)and(1.1,-1.5)..node[fill=white,inner sep=2]{$4$}(d);
 \fill(u)circle(0.03); \fill(c)circle (0.03); \fill(d)circle(0.03);
 \draw[->](d)++(120:0.3)arc[radius=0.3,start angle=120,end angle=60]node[above,pos=0.5]{$a_1$};
 \draw[->](d)++(50:0.3)arc[radius=0.3,start angle=50,end angle=15]node[right,pos=0.3]{$a_2$};
 \draw[->](d)++(5:0.3)arc[radius=0.3,start angle=5,end angle=-10]node[right,pos=0.5]{$a_3$};
 \draw[->](d)++(-20:0.3)arc[radius=0.3,start angle=-20,end angle=-160]node[above,pos=0.5]{$a_4$};
 \draw[->](d)++(-170:0.3)arc[radius=0.3,start angle=-170,end angle=-185]node[left,pos=0.5]{$a_5$};
 \draw[->](d)++(-195:0.3)arc[radius=0.3,start angle=-195,end angle=-230]node[left,pos=0.7]{$a_6$};
 \draw[->](c)++(150:0.2)arc[radius=0.2,start angle=150,end angle=30]node[below,pos=0.5]{$b$};
 \draw[->](u)++(-30:0.2)arc[radius=0.2,start angle=-30,end angle=-150]node[above,pos=0.5]{$c$};
\end{tikzpicture}\hspace{-3mm}
Q_T=
\begin{tikzpicture}[baseline=-1mm,scale=1.7]
 \node(3)at(0,0){$3$}; \node(1)at(-1,-1){$1$}; \node(2)at(1,-1){$2$}; \node(4)at(1,1){$4$}; \node(5)at(-1,1){$5$};
 \draw[->](3)--node[fill=white,inner sep=1]{$a_6$}(1); \draw[->](2)--node[fill=white,inner sep=1]{$a_2$}(3);
 \draw[->](3)--node[fill=white,inner sep=1]{$a_3$}(4); \draw[->](5)--node[fill=white,inner sep=1]{$a_5$}(3);
 \draw[transform canvas={yshift=2pt},->](1)--node[above]{$a_1$}(2);
 \draw[transform canvas={yshift=-2pt},->](1)--node[below]{$b$}(2);
 \draw[transform canvas={yshift=2pt},->](4)--node[above]{$c$}(5);
 \draw[transform canvas={yshift=-2pt},->](4)--node[below]{$a_4$}(5);
\end{tikzpicture}.
\]
We take a non-degenerate potential $W_T$ of $Q_T$ given in \cite{Lab09} (for a puncture $p$, we take the parameter $x_p=-1$) as follows:
\[
W_T=ba_6a_2+ca_3a_5-a_6a_5a_4a_3a_2a_1.
\]
Then the associated Jacobian algebra $J(T):=J(Q_T,W_T)$ is the quotient of the path algebra $KQ_T$ by the ideal generated by
\begin{eqnarray*}
&a_6a_2,\ a_3a_5,\ a_6a_5a_4a_3a_2,\ a_3a_2a_1a_6a_5, \\
&ba_6-a_1a_6a_5a_4a_3,\ a_2b-a_5a_4a_3a_2a_1,\ ca_3-a_4a_3a_2a_1a_6,\ a_5c-a_2a_1a_6a_5a_4.
\end{eqnarray*}
The indecomposable projectibe module $J(T)e_3$ is described by
\[
\begin{tikzpicture}[baseline=0mm]
 \node(s)at(0,0){$3$}; \node(t)at(0,-5){$3$}; 
 \path(s)++(-1,-0.5)node(l1){$1$}; \path(l1)++(0,-1)node(l2){$2$}; \path(l2)++(0,-1)node(l3){$3$}; \path(l3)++(0,-1)node(l4){$4$}; \path(l4)++(0,-1)node(l5){$5$};
 \path(s)++(1,-0.5)node(r1){$4$}; \path(r1)++(0,-1)node(r2){$5$}; \path(r2)++(0,-1)node(r3){$3$}; \path(r3)++(0,-1)node(r4){$1$}; \path(r4)++(0,-1)node(r5){$2$};
 \draw[->](s)--node[above,pos=0.6]{$a_6$}(l1); \draw[->](s)--node[above,pos=0.6]{$a_3$}(r1);
 \draw[->](l1)--node[left]{$a_1$}(l2); \draw[->](l2)--node[left]{$a_2$}(l3); \draw[->](l3)--node[left]{$a_3$}(l4); \draw[->](l4)--node[left]{$a_4$}(l5);
 \draw[->](r1)--node[right]{$a_4$}(r2); \draw[->](r2)--node[right]{$a_5$}(r3); \draw[->](r3)--node[right]{$a_6$}(r4); \draw[->](r4)--node[right]{$a_1$}(r5);
 \draw[->](l5)--node[below,pos=0.4]{$a_5$}(t); \draw[->](r5)--node[below,pos=0.4]{$a_2$}(t);
 \draw[->](l1)--node[right,pos=0.2]{$b$}(r5); \draw[->](r1)--node[left,pos=0.2]{$c$}(l5);
\end{tikzpicture}
\]
where all paths from the top vertex to each vertex coincide in $J(T)$ and the paths for all vertices form a basis of $J(T)e_3$. In particular, the dimension vector is given by the numbers of tagged arcs appearing in the diagram, that is, 
\[
\udim J(T)e_3=(2,2,4,2,2).
\]
Therefore, since $\dim\End_{J(T)}(J(T)e_3)=4\neq 1$, $Q_T$ does not satisfy \eqref{eq:d=dim} by Theorem \ref{thm:TFAE final}. Note that the dimension vector is also given by $\Int(T,\rho^{-1}(3))$ in the same way as the proof of Theorem \ref{thm:TFAE final}.

Next, we consider the tagged triangulation $\mu_3T$ and the associated quiver:
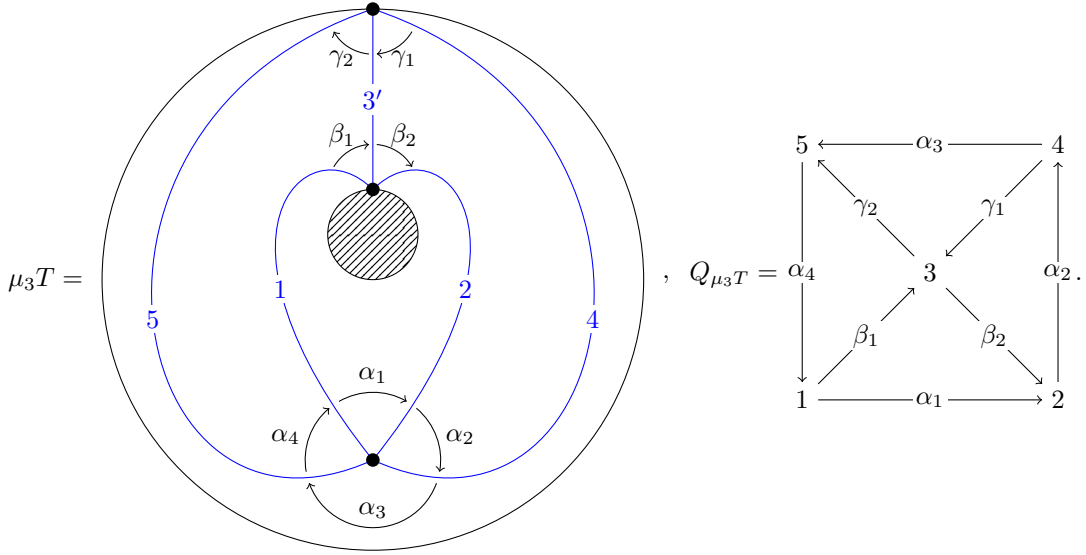
\[
\begin{tikzpicture}[baseline=-6mm,scale=3]
 \node at(-1.45,-0.2){$\mu_3T=$}; \node at(1.3,-0.2){$,$};
 \coordinate(0)at(0,0); \coordinate(u)at(0,1); \coordinate(c)at(0,0.2); \coordinate(d)at(0,-1);
 \draw[pattern=north east lines](0)circle(2mm); \draw(0,-0.2)circle(12mm);
 \draw[blue](d)..controls(-0.8,0)and(-0.3,0.5)..node[fill=white,inner sep=2,pos=0.3]{$1$}(c);
 \draw[blue](d)..controls(0.8,0)and(0.3,0.5)..node[fill=white,inner sep=2,pos=0.3]{$2$}(c);
 \draw[blue](c)--node[fill=white,inner sep=2]{$3'$}(u);
 \draw[blue](u)..controls(-1.5,0.5)and(-1.1,-1.5)..node[fill=white,inner sep=2]{$5$}(d);
 \draw[blue](u)..controls(1.5,0.5)and(1.1,-1.5)..node[fill=white,inner sep=2]{$4$}(d);
 \fill(u)circle(0.03); \fill(c)circle (0.03); \fill(d)circle(0.03);
 \draw[->](d)++(120:0.3)arc[radius=0.3,start angle=120,end angle=60]node[above,pos=0.5]{$\alpha_1$};
 \draw[->](d)++(50:0.3)arc[radius=0.3,start angle=50,end angle=-10]node[right,pos=0.5]{$\alpha_2$};
 \draw[->](d)++(-20:0.3)arc[radius=0.3,start angle=-20,end angle=-160]node[above,pos=0.5]{$\alpha_3$};
 \draw[->](d)++(-170:0.3)arc[radius=0.3,start angle=-170,end angle=-230]node[left,pos=0.5]{$\alpha_4$};
 \draw[->](c)++(150:0.2)arc[radius=0.2,start angle=150,end angle=95]node[above,pos=0.3]{$\beta_1$};
 \draw[->](c)++(85:0.2)arc[radius=0.2,start angle=85,end angle=30]node[above,pos=0.7]{$\beta_2$};
 \draw[->](u)++(-30:0.2)arc[radius=0.2,start angle=-30,end angle=-85]node[below,pos=0.3]{$\gamma_1$};
 \draw[->](u)++(-95:0.2)arc[radius=0.2,start angle=-95,end angle=-150]node[below,pos=0.7]{$\gamma_2$};
\end{tikzpicture}\hspace{-3mm}
Q_{\mu_3T}=
\begin{tikzpicture}[baseline=-1mm,scale=1.7]
 \node(3)at(0,0){$3$}; \node(1)at(-1,-1){$1$}; \node(2)at(1,-1){$2$}; \node(4)at(1,1){$4$}; \node(5)at(-1,1){$5$};
 \draw[<-](3)--node[fill=white,inner sep=1]{$\beta_1$}(1); \draw[->](3)--node[fill=white,inner sep=1]{$\beta_2$}(2);
 \draw[<-](3)--node[fill=white,inner sep=1]{$\gamma_1$}(4); \draw[->](3)--node[fill=white,inner sep=1]{$\gamma_2$}(5);
 \draw[->](1)--node[fill=white,inner sep=1]{$\alpha_1$}(2); \draw[->](2)--node[fill=white,inner sep=1]{$\alpha_2$}(4);
 \draw[->](4)--node[fill=white,inner sep=1]{$\alpha_3$}(5); \draw[->](5)--node[fill=white,inner sep=1]{$\alpha_4$}(1);
\end{tikzpicture}.
\]
We take a non-degenerate potential $W_{\mu_3T}$ of $Q_{\mu_3T}$ as follows:
\[
W_{\mu_3T}=\gamma_2\beta_1\alpha_4+\beta_2\gamma_1\alpha_2+\alpha_4\alpha_3\alpha_2\alpha_1.
\]
Note that the quivers with potential $(Q_{\mu_3T},W_{\mu_3T})$ and $\mu_3(Q_T,W_T)$ are right-equivalent \cite[Theorem 30]{Lab09}. Then the associated Jacobian algebra $J(\mu_3T):=J(Q_{\mu_3T},W_{\mu_3T})$ is the quotient of the path algebra $KQ_{\mu_3T}$ by the ideal generated by
\begin{eqnarray*}
&\alpha_4\gamma_2,\ \beta_1\alpha_4,\ \alpha_2\beta_2,\ \gamma_1\alpha_2,\ \alpha_4\alpha_3\alpha_2,\ \alpha_2\alpha_1\alpha_4\\
&\gamma_2\beta_1-\alpha_3\alpha_2\alpha_1,\ \beta_2\gamma_1-\alpha_1\alpha_4\alpha_3.
\end{eqnarray*}
All indecomposable projectibe $J(\mu_3T)$-modules are described as follows:
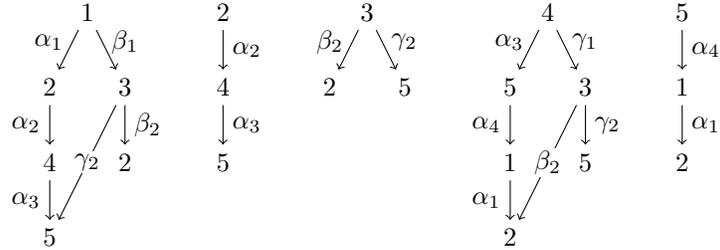
\[
\begin{tikzpicture}[baseline=0mm]
 \node(s)at(0,0){$1$};
 \path(s)++(-0.5,-1)node(l1){$2$}; \path(l1)++(0,-1)node(l2){$4$}; \path(l2)++(0,-1)node(l3){$5$};
 \path(s)++(0.5,-1)node(r1){$3$}; \path(r1)++(0,-1)node(r2){$2$};
 \draw[->](s)--node[left,pos=0.3]{$\alpha_1$}(l1); \draw[->](s)--node[right,pos=0.3]{$\beta_1$}(r1);
 \draw[->](l1)--node[left]{$\alpha_2$}(l2); \draw[->](l2)--node[left]{$\alpha_3$}(l3);
 \draw[->](r1)--node[right]{$\beta_2$}(r2);
 \draw[->](r1)--node[fill=white,inner sep=1]{$\gamma_2$}(l3);
\end{tikzpicture}
\hspace{5mm}
\begin{tikzpicture}[baseline=0mm]
 \node(s)at(0,0){$2$}; 
 \path(s)++(0,-1)node(l1){$4$}; \path(l1)++(0,-1)node(l2){$5$};
 \draw[->](s)--node[right]{$\alpha_2$}(l1); \draw[->](l1)--node[right]{$\alpha_3$}(l2);
\end{tikzpicture}
\hspace{5mm}
\begin{tikzpicture}[baseline=0mm]
 \node(s)at(0,0){$3$};
 \path(s)++(-0.5,-1)node(l1){$2$}; \path(s)++(0.5,-1)node(r1){$5$};
 \draw[->](s)--node[left,pos=0.3]{$\beta_2$}(l1); \draw[->](s)--node[right,pos=0.3]{$\gamma_2$}(r1);
\end{tikzpicture}
\hspace{5mm}
\begin{tikzpicture}[baseline=0mm]
 \node(s)at(0,0){$4$};
 \path(s)++(-0.5,-1)node(l1){$5$}; \path(l1)++(0,-1)node(l2){$1$}; \path(l2)++(0,-1)node(l3){$2$};
 \path(s)++(0.5,-1)node(r1){$3$}; \path(r1)++(0,-1)node(r2){$5$};
 \draw[->](s)--node[left,pos=0.3]{$\alpha_3$}(l1); \draw[->](s)--node[right,pos=0.3]{$\gamma_1$}(r1);
 \draw[->](l1)--node[left]{$\alpha_4$}(l2); \draw[->](l2)--node[left]{$\alpha_1$}(l3);
 \draw[->](r1)--node[right]{$\gamma_2$}(r2);
 \draw[->](r1)--node[fill=white,inner sep=1]{$\beta_2$}(l3);
\end{tikzpicture}
\hspace{5mm}
\begin{tikzpicture}[baseline=0mm]
 \node(s)at(0,0){$5$}; 
 \path(s)++(0,-1)node(l1){$1$}; \path(l1)++(0,-1)node(l2){$2$};
 \draw[->](s)--node[right]{$\alpha_4$}(l1); \draw[->](l1)--node[right]{$\alpha_1$}(l2);
\end{tikzpicture}
\]
We can see that $\dim\End(J(\mu_3T)e_i)=1$ for any $1 \le i \le 5$ and $Q_{\mu_3T}$ satisfies \eqref{eq:d=dim} by Theorem \ref{thm:TFAE final}.

\medskip\noindent{\bf Acknowledgements}.
The author would like to thank Pierre-Guy Plamondon for helpful comments on Proposition \ref{prop:fin dim}.
He was JSPS Overseas Research Fellow and supported by JSPS KAKENHI Grant Numbers JP20J00410, JP21K13761.

\bibliographystyle{alpha}
\bibliography{d=dim}

\end{document}